\documentclass[a4paper,11pt,UKenglish]{article}

\usepackage{latexsym}
\usepackage{amssymb}
\usepackage{amsfonts}
\usepackage{amsmath,amsthm}
\usepackage{xypic}
\usepackage[english]{babel}

\usepackage[top=3cm, bottom=3.1cm, left=3.8cm, right=3.8cm]{geometry}
\usepackage{caption}
\usepackage{tikz}
\usetikzlibrary{patterns, matrix}
\tikzstyle{NE-lines}=[pattern=north east lines, pattern color=black!45]

\newcommand{\Av}{\mathrm{Av}}

\newcommand{\Sort}{\mathrm{Sort}}

\newcommand{\out}{s_{\sigma}}
\newcommand{\U}{\mathtt{U}} 
\newcommand{\D}{\mathtt{D}} 
\renewcommand{\H}{\mathtt{H}} 
\newcommand{\Do}{\texttt{min}}
\newcommand{\Co}{\texttt{cons}}
\newcommand{\Mesh}{(132, \left\lbrace (0,2),(2,0),(2,1) \right\rbrace)}
\newcommand{\M}{\mathcal{M}}
\newcommand{\meshpatt}{\mu}
\newcommand{\rmost}{\mathrm{rm}}
\newcommand{\lmost}{\mathrm{lm}}
\renewcommand{\top}{\mathrm{top}}
\newcommand{\std}{\mathrm{std}}

\newtheorem{theorem}{Theorem}[section]
\newtheorem{prop}[theorem]{Proposition}
\newtheorem{lemma}[theorem]{Lemma}
\newtheorem{cor}[theorem]{Corollary}
\newtheorem{openpr}[theorem]{Open Problem}

\theoremstyle{definition}
\newtheorem{remark}[theorem]{Remark}

\newcommand\rgf{{\sc rgf}}
\newcommand\rgfs{{\sc rgf}s}

\author{Giulio Cerbai\thanks{Dipartimento di Matematica e Informatica
    ``U.  Dini'', University of Firenze, Firenze, Italy,
    \tt{giulio.cerbai@unifi.it, luca.ferrari@unifi.it}}
  \and Anders Claesson\thanks{Science Institute, University of Iceland, Iceland,
    \tt{akc@hi.is}}
  \and Luca Ferrari$^{\dag}$
  \and Einar Steingr\'{\i}msson$^{\dag ,\!\!}$\thanks{Department of Mathematics and
    Statistics, University of Strathclyde, Glasgow, Scotland,
    {\tt{einar@alum.mit.edu}}.}}

\title{Sorting with pattern-avoiding stacks:\\ the $132$-machine\footnote{G.C. and L.F. are members of the INdAM Research group GNCS;
 they are partially supported by INdAM-GNCS 2020 project
 ``Combinatoria delle permutazioni, delle parole e dei grafi:
 algoritmi e applicazioni''.  E.S. was partially supported by a Leverhulme Research Fellowship.}}

\begin{document}

\maketitle

\begin{abstract}
  This paper continues the analysis of the pattern-avoiding sorting machines
  recently introduced by Cerbai, Claesson and Ferrari~\cite{CCF}.  These
  devices consist of two stacks, through which a permutation is passed in
  order to sort it, where the content of each stack must at all times avoid a
  certain pattern.  Here we characterize and enumerate the set of
  permutations that can be sorted when the first stack is $132$-avoiding,
  solving one of the open problems proposed in~\cite{CCF}. To that end we
  present several connections with other well known combinatorial objects,
  such as lattice paths and restricted growth functions (which encode set
  partitions). We also provide new proofs for the enumeration of some sets of
  pattern-avoiding restricted growth functions and we expect that the tools
  introduced can be fruitfully employed to get further similar results.
\end{abstract}

\section{Introduction}
\thispagestyle{empty}

Pattern-avoiding sorting machines were introduced in a recent paper by
Cerbai, Claesson and Ferrari~\cite{CCF} aiming towards a better understanding
of the problem of sorting permutations with two stacks in series. In the
classical formulation of the Stacksort problem~\cite{Kn}, an input
permutation $\pi=\pi_1\ldots\pi_n$ is scanned from left to right and, when
$\pi_i$ is the current element, either $\pi_i$ is pushed onto the stack or
the top element of the stack is popped and appended to the output. If there 
is a sequence of push and pop operations that produces a sorted output (that
is, the identity permutation), then the input permutation is said to be
\emph{sortable}. There is a well known algorithm, called \emph{Stacksort},
that sorts every sortable permutation. It has two key properties:
\begin{enumerate}
\item the stack is \emph{increasing}, meaning that the elements inside
  the stack are maintained in increasing order (from top to bottom);
\item the algorithm is \emph{right greedy}, meaning that it always
  chooses to perform a push operation as long as the stack remains
  increasing in the above sense; here the expression ``right greedy''
  refers to the usual pictorial representation of this problem, in which
  the input permutation is on the right, the stack is in the middle and
  the output permutation is on the left (see
  Figure~\ref{stacksort_machine}, left).
\end{enumerate}

The notion of pattern avoidance allows us to efficiently characterize the set
of the permutations that can be sorted by \emph{Stacksort}. Let $S_n$ be the
symmetric group over a set of cardinality $n$, consisting of all permutations
of length $n$. Given two permutations $\sigma \in S_k$ and
$\pi =\pi_1 \cdots \pi_n \in S_n$, with $k\leq n$, we say that $\sigma$ is a
\emph{pattern} of $\pi$ when there exist indices
$1\leq i_1 <i_2 <\cdots <i_k \leq n$ such that
$\pi_{i_1}\pi_{i_2}\ldots \pi_{i_k}$ (as a permutation) is isomorphic to
$\sigma$, that is, $\pi_{i_1},\pi_{i_2},\ldots,\pi_{i_k}$ are in the same
relative order of size as the elements of $\sigma$, in which case we write
$\sigma \simeq \pi_{i_1}\pi_{i_2}\ldots \pi_{i_k}$. This notion of patterns
in permutations defines a partial order, and the resulting poset is known as
the \emph{permutation pattern poset}. When $\sigma$ is a pattern of $\pi$, we 
say that $\pi$ \emph{contains} $\sigma$, otherwise $\pi$ \emph{avoids}
$\sigma$. A downset $I$ of the permutation pattern poset, also called a 
\emph{permutation class}, can be described in terms of its minimal excluded
permutations (or, equivalently, the minimal elements of the complementary
upset); these permutations are called the \emph{basis} of $I$. When $B$ is
the basis of $I$ we write $I=\Av(B)$.

Returning to \emph{Stacksort}, it is well known that a permutation is
sortable if and only if it avoids the pattern $231$. As a consequence,
the number of sortable permutations of length $n$ is the $n$-th Catalan
number. Given that describing the set of sortable permutations is rather
manageable in the classical case, one would think that similar results
can be derived by considering a slightly more general version of the
problem, where a second stack is connected in series to the first
one. Despite the many attempts, very few results have been obtained. For
example, Murphy~\cite{M} showed that thus sortable permutations are a class
with infinite basis. To describe the basis and to enumerate the
permutations in question remain open problems.

Due to the toughness of the problem in its full generality, several authors
have considered weaker formulations by introducing some constraints on the
sorting device. In his PhD thesis~\cite{W}, West studied permutations that
can be sorted by two stacks connected in series using a \emph{right greedy
  algorithm}. This is equivalent to making two passes through a
stack. Similarly, Smith~\cite{Sm} considered two stacks in series, where the
first stack is required to be \emph{decreasing}. It is worth noting that, due
to the properties of classical stacksort, the second (final) stack turns out
to be necessarily increasing.

\begin{figure}
  \ \hspace{-2ex}\includegraphics[viewport=0 700 500 800]{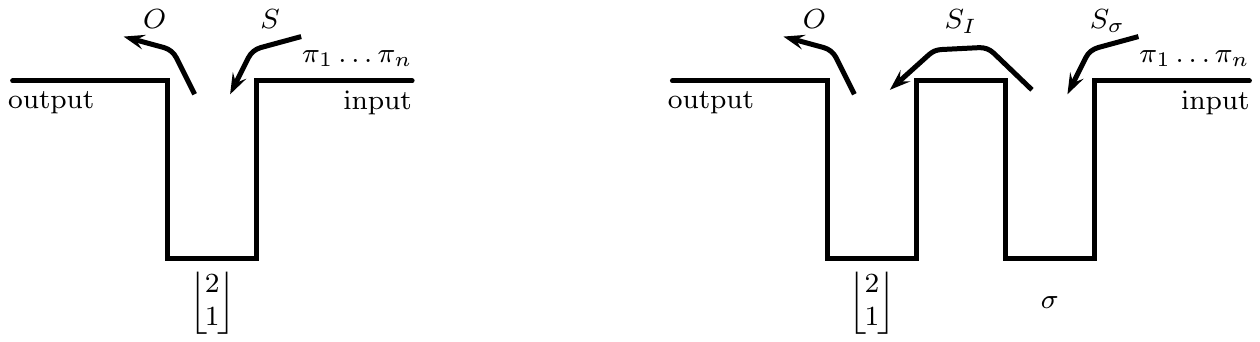}
  \caption{Sorting with one stack (left) and sorting with the
    $\sigma$-machine (right).}\label{stacksort_machine}
\end{figure}

Pattern-avoiding machines constitute a further proposal to approach the
general problem of sorting with two stacks. Let $\sigma$ be a
permutation. The \emph{$\sigma$-machine} consists of two stacks
connected in series (see Figure~\ref{stacksort_machine}, right),
obeying the following constraints:

\begin{enumerate}
\item At each step of the procedure, the elements in each stack must
  avoid certain forbidden configurations, reading from top to
  bottom. The second stack is increasing, that is, the sequence of numbers
  contained in the stack has to avoid the pattern $21$. We express
  this by saying that the stack is $21$-avoiding. In the same spirit,
  the first stack is \emph{$\sigma$-avoiding}.
\item The algorithm performed with the two stacks connected in series is
  \emph{right greedy}. As already observed, this is equivalent to making
  two passes through a stack, performing the right greedy algorithm at
  each pass. However, due to the restriction described above, during the
  first pass the stack is $\sigma$-avoiding, whereas during the second
  pass it is $21$-avoiding.
\end{enumerate}

We refer to the $\sigma$-avoiding stack as the \textit{$\sigma$-stack}.
A permutation $\pi$ is \emph{$\sigma$-sortable} if it is sortable by the
$\sigma$-machine. Denote by $\Sort(\sigma)$ the set of $\sigma$-sortable
permutations and by $\Sort_n(\sigma)$ the set of $\sigma$-sortable
permutations of length $n$. Denote by $s_{\sigma}(\pi )$ the output of
the $\sigma$-stack on input $\pi$. Observe that, since
$s_{\sigma}(\pi)$ is the input to the second (classical) stack, a
permutation $\pi$ is $\sigma$-sortable if and only if $s_{\sigma}(\pi)$
avoids $231$. This fact, which will be frequently used throughout the
paper, allows us to restrict our attention to the behavior of the
$\sigma$-stack when analyzing the sortability of $\pi$.

In~\cite{CCF}, the authors determine the patterns $\sigma$ such that
$\Sort(\sigma)$ is a permutation class, providing explicitly the
corresponding basis.

\begin{theorem}[\cite{CCF}, Theorems 3.2 and 3.4]\label{classes_and_non_classes}
  Let $\sigma=\sigma_1 \sigma_2 \sigma_3 \cdots \sigma_k$ and let
  $\widehat{\sigma}=\sigma_2 \sigma_1 \sigma_3 \cdots \sigma_k$ be the
  permutation obtained by exchanging the first two elements of
  $\sigma$. Then:
  \begin{enumerate}
  \item $\Sort(\sigma)$ is a permutation class if and only if $\widehat\sigma$ contains $231$.
  \item If $\widehat{\sigma}$ contains $231$, then
    $\Sort(\sigma )=\Av(132,\sigma^r)$, where
    $\sigma^r=\sigma_k \cdots \sigma_2 \sigma_1$.
  \end{enumerate}
\end{theorem}

Theorem~\ref{classes_and_non_classes} completely describes the sets of
$\sigma$-sortable permutations that are permutation classes. The
remaining cases are much more challenging. For example, amongst the six
permutations of length three, $\Sort(321)=\Av(123,132)$ as a consequence
of the previous result, but so far the only other solved pattern is
$123$: $123$-sortable permutations are shown to be enumerated by the
partial sums of partial sums of the Catalan numbers (sequence A294790
in~\cite{Sl}) via a bijection with Schr\"oder paths avoiding the pattern 
$\U\H\D$~\cite{CF}. In this paper we deal with one of the remaining patterns of
$S_3$, namely $132$.

In Section~\ref{section_mesh} we characterize $132$-sortable
permutations as those avoiding the classical pattern $2314$ and a
certain mesh pattern.

In Section~\ref{section_grid} we exploit the pattern avoidance
characterization of $\Sort(132)$ to provide a geometrical description of
these permutations. This ultimately allows us to find a recursive
construction for $\Sort(132)$, which is used to provide a bijection between
$\Sort(132)$ and the set of restricted growth functions (\rgfs, to be defined
in next section) avoiding the pattern $12231$. The enumeration of the
12231-avoiding {\rgfs} was obtained by Jel\'{\i}nek and Mansour in~\cite{JM},
where they present a much more general mechanism that determines the entire
Wilf-equivalence class of these avoiders, that is, the class of patterns that
are avoided by the same number of {\rgfs} of each length $n$.  Their counting
sequence is the binomial transform of the Catalan numbers, which is A007317
in the OEIS~\cite{Sl}.

In Section~\ref{section_enum} we exhibit direct combinatorial proofs for
the enumeration of some patterns in the same Wilf-equivalence class as
$12231$. We exhibit links with lattice paths and pattern-avoiding 
permutations. Two of these patterns are enumerated via a bijection with
a family of labeled Motzkin paths, which provides a natural
combinatorial interpretation for a beautiful continued fraction for
A007317. We also conjecture that a slight variation on the same approach
should lead to the enumeration of many other patterns in the same
Wilf-class. Finally, some of the results in this section lead to an
independent proof of the enumeration of $\Sort(132)$.

\section{Preliminaries and notation}\label{prelim}

Given a permutation $\pi=\pi_1\pi_2\ldots\pi_n$, the element $\pi_i$ is
called a \emph{left-to-right maximum} (briefly, \emph{ltr-maximum}) if 
$\pi_i > \max \left\lbrace \pi_1 ,\ldots ,\pi_{i-1} \right\rbrace$.
Analogously, $\pi_i$ is called a \emph{ltr-minimum} if
$\pi_i < \min \left\lbrace \pi_1 ,\ldots ,\pi_{i-1} \right\rbrace$. The
element $\pi_1$ is both an ltr-maximum and ltr-minimum.  A \emph{descent} of 
$\pi$ is a pair of elements $(\pi_i, \pi_{i+1})$ such that
$\pi_i > \pi_{i+1}$.  This is a slight deviation from the classical
definition, in which a descent is an index $i$ such that $\pi_i>\pi_{i+1}$. A
descent is said to be \emph{consecutive} if
$\pi_{i+1}=\pi_{i}-1$. \emph{Ascents} and \emph{consecutive ascents} are
defined similarly. For example, the permutation $\pi=3417625$ has three
ltr-maxima, namely $3,4,7$ and two ltr-minima $3,1$. The descents of $\pi$ are
$(4,1),(7,6),(6,2)$, where only $(7,6)$ is a consecutive descent. The ascents
are $(3,4),(1,7),(2,5)$ and only $(3,4)$ is consecutive.

Given two permutations $\alpha=\alpha_1 \ldots \alpha_n$ and
$\beta=\beta_1 \ldots \beta_m$, the \emph{direct sum} $\alpha \oplus \beta$
is the permutation $\pi=\pi_1 \ldots \pi_n \pi_{n+1} \ldots \pi_{n+m}$ of
length $n+m$ such that $\pi_1 \ldots \pi_n \simeq \alpha$,
$\pi_{n+1} \ldots \pi_{n+m} \simeq \beta$ and $\pi_i < \pi_j$, for each
$i \in \lbrace 1,\dots,n \rbrace$ and $j \in \lbrace n+1, \dots,n+m
\rbrace$.
The \emph{skew sum} $\alpha \ominus \beta$ is defined similarly, but
requiring that $\pi_i > \pi_j$ for each $i \in \lbrace 1,\dots,n \rbrace$ and
$j \in \lbrace n+1, \dots,n+m \rbrace$. For example, $213\oplus21=21354$ and  
$213\ominus21=43521$. A permutation is said to be \emph{layered} if it is
the direct sum of decreasing permutations. It is well known that $\pi$ is
layered if and only if $\pi \in \Av(231,312)$ and there are $2^{n-1}$ layered
permutations of length~$n$.

A \emph{Dyck path} is a path in the discrete plane
$\mathbb{Z}\times \mathbb{Z}$ starting at the origin of a fixed
Cartesian coordinate system, ending on the $x$-axis, never falling below
the $x$-axis and using two kinds of steps, namely
upsteps $\U=(1,1)$ and downsteps $\D=(1,-1)$. The \emph{length} of a Dyck path
is its final abscissa, which coincides with the total number of its
steps. See Figure~\ref{Dyck_and_mesh_figure} for an example of Dyck
path. According to their semilength, Dyck paths are counted by Catalan
numbers (sequence A000108 in~\cite{Sl}). The $n$-th \emph{Catalan number} is
$\mathfrak{c}_n =\frac{1}{n+1}{2n\choose n}$ and the associated ordinary 
generating function is $C(x)=(1-\sqrt{1-4x})/(2x)$. A slightly more general
notion of lattice path is obtained by allowing one more kind of step, the 
horizontal step $\H=(1,0)$. The resulting paths are called \emph{Motzkin
  paths} and their enumeration (with respect to the total number of steps) is
given by the Motzkin numbers (sequence A001006 in~\cite{Sl}).

A \emph{Restricted Growth Function} ({\rgf}) of length $n$ is a sequence of
positive integers $R=r_1 \cdots r_n$ such that $r_1=1$ and
$r_i \le 1+\max \left\lbrace r_1,\dots,r_{i-1} \right\rbrace$ for each
$i \ge 2$. The {\rgfs} of length $n$ bijectively encode set partitions of
$[n]=\{1,2,\ldots,n\}$, where, for example, the partition of $[5]$ written in
standard notation as 13--25--4 has {\rgf} 12132, whose 3 in place 4 indicates
that 4 is in the third block.

Denote by $\mathcal{R}_n$ the set of {\rgfs} of length $n$ and let
$\mathcal{R}=\bigcup_{n \ge 1} \mathcal{R}_n$. The notion of pattern
avoidance can be naturally extended to {\rgfs}. Given a sequence of positive
integers $Q=q_1 q_2 \cdots q_k$, define the \emph{standardization} of $Q$ as
the string $\std(Q)$ obtained by replacing all occurrences of the $i$-th
smallest element with~$i$, for all $i$. Then, given a {\rgf}
$R=r_1 \ldots r_n$ and a sequence of positive integers $Q=q_1 \ldots q_k$,
with $k \le n$, $Q$ is a pattern of $R$ if there is a subsequence
$r_{i_1} \ldots r_{i_k}$ of $R$ such that $\std(r_{i_1} \ldots r_{i_k})=Q$.
In this case we write $Q \le R$ (and say that $R$ \emph{contains} $Q$);
otherwise, we say that $R$ \emph{avoids} $Q$. We use the notation
$\mathcal{R}(Q)$ to denote the set of the {\rgfs} avoiding $Q$ and
$\mathcal{R}_n(Q)=\mathcal{R}_n \cap \mathcal{R}(Q)$. For a more detailed
survey on the notion of pattern avoidance in {\rgfs}, we refer the reader
to~\cite{JM} and~\cite{CDDGGPS}. Observe that if $R$ is a {\rgf} then each
occurrence of the integer $k$ in $R$, for any $k \ge 1$, is preceded by some
occurrence of all the integers $1,\dots,k-1$. A useful consequence is the
following lemma, whose easy proof is omitted.

\begin{lemma}\label{RGF_prop}
Let $R$ be a {\rgf} and let $Q=q_1 q_2 \ldots q_k$ be a sequence of positive
integers. Let $Q'=\std(Q)=q'_1 \ldots q'_k$ and suppose that $q'_1=t$, for some $t\ge 1$.
Then $Q'\le R$ if and only if $12\dots(t-1)Q'\le R$.
\end{lemma}

\section{Pattern avoidance characterization of
  $\Sort(132)$}\label{section_mesh}

\emph{For the remainder of this paper, we let $\sigma=132$.}\medskip

In this section we characterize $\Sort(\sigma)$ in terms of pattern
avoidance. First we need to introduce a slightly more general notion of
pattern, originally given by Br\"and\'en and Claesson in~\cite{BC}. A
\emph{mesh pattern} of length $k$ is a pair $(\tau,A)$, where
$\tau \in S_k$ and
$A \subseteq \left[ 0,k \right] \times \left[ 0,k \right]$ is a set of
pairs of integers. The elements of $A$ identify the lower left corners
of forbidden squares in the plot of $\tau$ (see
Figure~\ref{Dyck_and_mesh_figure}). An occurrence of the mesh pattern
$(\tau,A)$ in $\pi$ is then an occurrence of the classical pattern
$\tau$ in $\pi$ such that no elements of $\pi$ are placed into a
forbidden square of $A$.

\begin{figure}
  \begin{center}
    \begin{tikzpicture}[scale=0.5, baseline=20.5pt]
      \draw [ultra thin] (0,0) -- (18,0);
      \draw [semithick] (0,0) -- (1,1);
      \draw [semithick] (1,1) -- (2,0);
      \draw [semithick] (2,0) -- (6,4);
      \draw [semithick] (6,4) -- (8,2);
      \draw [semithick] (8,2) -- (9,3);
      \draw [semithick] (8,2) -- (9,3);
      \draw [semithick] (9,3) -- (10,2);
      \draw [semithick] (10,2) -- (11,3);
      \draw [semithick] (11,3) -- (14,0);
      \draw [semithick] (14,0) -- (16,2);
      \draw [semithick] (16,2) -- (18,0);
      \filldraw(0,0) circle (6pt);
      \filldraw(1,1) circle (6pt);
      \filldraw(2,0) circle (6pt);
      \filldraw(3,1) circle (6pt);
      \filldraw(4,2) circle (6pt);
      \filldraw(5,3) circle (6pt);
      \filldraw(6,4) circle (6pt);
      \filldraw(7,3) circle (6pt);
      \filldraw(8,2) circle (6pt);
      \filldraw(9,3) circle (6pt);
      \filldraw(10,2) circle (6pt);
      \filldraw(11,3) circle (6pt);
      \filldraw(12,2) circle (6pt);
      \filldraw(13,1) circle (6pt);
      \filldraw(14,0) circle (6pt);
      \filldraw(15,1) circle (6pt);
      \filldraw(16,2) circle (6pt);
      \filldraw(17,1) circle (6pt);
      \filldraw(18,0) circle (6pt);
      \node[above,left] at (0.75,0.75) {$1$};
      \node[above,left] at (2.75,0.75) {$1$};
      \node[above,left] at (3.75,1.75) {$2$};
      \node[above,left] at (4.75,2.75) {$3$};
      \node[above,left] at (5.75,3.75) {$4$};
      \node[above,left] at (8.75,2.75) {$2$};
      \node[above,left] at (10.75,2.75) {$2$};
      \node[above,left] at (14.75,0.75) {$4$};
      \node[above,left] at (15.75,1.75) {$5$};
    \end{tikzpicture}
    \hspace{1cm}
    \begin{tikzpicture}[scale=0.5, baseline=20.5pt]
      \fill[NE-lines] (0,2) rectangle (1,3);
      \fill[NE-lines] (2,0) rectangle (3,1);
      \fill[NE-lines] (2,1) rectangle (3,2);
      \draw [semithick] (0.001,0.001) grid (3.999,3.999);
      \filldraw (1,1) circle (6pt);
      \filldraw (2,3) circle (6pt);
      \filldraw (3,2) circle (6pt);
    \end{tikzpicture}
  \end{center}
  \caption{A Dyck path (on the left) and the mesh pattern
    $\meshpatt=\Mesh$ (on the right).}\label{Dyck_and_mesh_figure}
\end{figure}

We start by proving a useful decomposition lemma for $\sigma$-sortable
permutations.  Given a permutation $\pi$ we decompose it as
$\pi=m_1 B_1 m_2 B_2 \ldots m_k B_k$, where
$m_1 \ge m_2 \ge \cdots \ge m_k=1$ are the ltr-minima of $\pi$ and each block 
$B_i$ contains all the elements strictly between two consecutive ltr-minima.
We refer to this as the \emph{ltr-minima decomposition} of $\pi$.

\begin{lemma}\label{ltr_min_sortable}
  Let $\pi$ be a permutation and let
  $\pi=m_1 B_1 m_2 B_2 \ldots m_k B_k$ be its ltr-minima
  decomposition. Then:
  \begin{enumerate}
  \item
    $\out(\pi)=\widetilde{B_1} \widetilde{B_2}\cdots \widetilde{B_k} m_k m_{k-1}
    \cdots m_2 m_1$, where each $\widetilde{B_i}$ is a suitable rearrangement of
    the elements of $B_i$.
  \item If $\pi$ is $\sigma$-sortable, then $x>y$ for each $x \in B_i$,
    $y \in B_j$, with $i<j$.
  \end{enumerate}
\end{lemma}

\begin{proof}
  \begin{enumerate}
  \item For each $x \in B_1$, $m_1 x m_2 \simeq 231$, thus every element of
    $B_1$ has to be popped from the $\sigma$-stack before $m_2$ enters. After
    that, we have $m_1$ and $m_2$ on the $\sigma$-stack, with $m_1 > m_2$ and
    $m_2$ above $m_1$. Note that they cannot both be part of a $132$,
    therefore $m_2$ remains on the $\sigma$-stack until the end of the
    sorting process. Similarly, each element of $B_2$ has to be popped
    before~$m_3$ enters, since $m_3xm_2\simeq 132$ for each $x \in B_2$. The
    same argument holds for every $m_j$ with $j\ge2$.
  \item Suppose there are two elements $x,y$ such that $x<y$,
    $x \in B_i$ and $y \in B_j$, with $i<j$. Then, as a consequence of
    the previous item, $x y m_k$ is an occurrence of $231$ in
    $\out{(\pi)}$, which is a contradiction since $\pi$ is $\sigma$-sortable.\qedhere
  \end{enumerate}
\end{proof}

\begin{lemma}\label{unimodal_stack}
  Let $\pi \in \Sort_n(\sigma)$ and let
  $\pi=m_1 B_1 m_2 B_2 \cdots m_k B_k$ be its ltr-minima
  decomposition. Then, when the next element of the input is
  $b \in B_i$ the content of the $\sigma$-stack when read from 
  bottom to top is $m_1 m_2 \cdots m_i b_1 b_2 \cdots b_t$, where 
  $\{ b_1 ,\ldots b_t \}$ is a (possibly empty) subset of $B_i$ such
  that $b_1 < b_2 < \cdots < b_t$.
\end{lemma}

\begin{proof}
  The first $i$ ltr-minima $m_1,\dots,m_i$ of $\pi$ lie at the bottom of the
  $\sigma$-stack, by Lemma~\ref{ltr_min_sortable}. Then the remaining
  elements $b_1,\dots,b_t$ of $B_i$ in the $\sigma$-stack must be in increasing order
  from bottom to top, for otherwise, if $b_h > b_\ell$ for some $h<\ell$,
  then $\out{(\pi)}$ would contain $b_\ell b_h m_i \simeq 231$, contradicting
  the $\sigma$-sortability of $\pi$.
\end{proof}

We next show that $\sigma$-sortable permutations are characterized by
the avoidance of a classical pattern and a mesh pattern.  This leads to
a more precise geometrical description of these permutations, as we will
show in the next section. For the rest of the paper, let
$\meshpatt=\Mesh$ be the mesh pattern depicted in
Figure~\ref{Dyck_and_mesh_figure}. An occurrence of the mesh pattern
$\meshpatt$ is thus an occurrence $acb$ of the classical pattern $132$
such that:
\begin{itemize}
\item every element that precedes $a$ in $\pi$ is either smaller than
  $b$ or greater than~$c$;
\item every element between $c$ and $b$ in $\pi$ is greater than $b$.
\end{itemize}

\begin{theorem}\label{mesh_patterns_necessary}
  If $\pi$ is $\sigma$-sortable, then $\pi \in \Av(2314,\meshpatt)$.
\end{theorem}

\begin{proof}
  Let $\pi=m_1 B_1 m_2 B_2 \cdots m_k B_k$ be the ltr-minima decomposition of
  $\pi$. Suppose, for a contradiction, that $\pi$ contains an occurrence
  $bcad$ of $2314$. When $a$ enters the $\sigma$-stack, at least one element
  between $b$ and $c$, call it $x$, has already been popped from the
  $\sigma$-stack, otherwise we would get the forbidden pattern
  $acb \simeq 132$ inside the
  $\sigma$-stack. 
  Hence, by Lemma~\ref{ltr_min_sortable}, $\out{(\pi)}$ contains
  $x d m_k \simeq 231$, violating the hypothesis that $\pi$ is
  $\sigma$-sortable.

  Next suppose that $acb$ is an occurrence of $132$ in $\pi$. We wish to show
  that $acb$ is part of an occurrence of either $\mathbf{3}142$,
  $24\mathbf{1}3$ or $14\mathbf{2}3$, thus proving that $\pi$ avoids the mesh
  pattern $\meshpatt$. Let $m(a)$ be the ltr-minimum of the block that
  contains $a$ (in particular, $m(a)=a$ if $a$ is a ltr-minimum itself). Then
  $m(a) \le a$ and $m(a)$ exits the $\sigma$-stack after $b$ and $c$ (by
  Lemma~\ref{ltr_min_sortable}), so $c$ has to be popped before $b$ enters,
  otherwise $b c m(a)$ would be an occurrence of $231$ inside
  $\out{(\pi)}$. We consider the following two cases. Note that $a<b<c$, so
  $b,c$ are not ltr-minima in $\pi$.
  \begin{itemize}
  \item $c \in B_i$ and $b \in B_j$, with $i < j$. In this case,
    $m_j < m(a) \le a$, hence $a c m_j b \simeq 2413$, which is one of
    the desired patterns.
  \item $c$ and $b$ are in the same block $B_i$. First suppose there is a
    ltr-minimum $m=m_\ell$, with $\ell<i$, such that $b<m<c$; then $m>m(a)$,
    so $m$ precedes $m(a)$ in $\pi$ and $macb \simeq 3142$, again one of the
    listed patterns. Otherwise, suppose that, for every ltr-minimum $m$,
    either $m<b$ or $m>c$ and consider the element $w$ that immediately
    precedes $b$ in $\pi$. We wish to show that $w<b$, which will conclude
    the proof. Suppose, for a contradiction, that $w>b$ and let
    $x_1,x_2,\dots,x_s=w$ be the elements on the $\sigma$-stack, after $w$
    has been pushed, that are not ltr-minima when we read from bottom to
    top. By Lemma~\ref{unimodal_stack}, we have $x_1<x_2<\cdots<x_s$;
    moreover $x_s=w > b$, so there is a minimum index $t$ such that $x_t
    >b$.
    Now observe that, for $\ell>t$, all the elements $x_\ell$ are popped from
    the $\sigma$-stack before $b$ enters, because $b x_\ell x_t \simeq 132$.
    We also observe that necessarily $x_t \le c$, otherwise $c$ would already
    have been popped and $\out{(\pi)}$ would contain the pattern
    $c x_t m(a) \simeq 231$. We can now assert that $b$ is pushed onto the
    $\sigma$-stack immediately above $x_t$. In fact, $x_\ell < b$ for every
    $\ell < t$; moreover, our hypothesis implies that either $m<b$ or $m>c$
    for every ltr-minimum $m$ inside the $\sigma$-stack, therefore $b$ cannot
    be the first element of an occurrence of $231$ (read from top to bottom)
    that involves elements inside the $\sigma$-stack. However this results in
    an occurrence $b x_t m(a)$ of $231$ in $\out{(\pi)}$, which again
    contradicts the hypothesis that $\pi$ is $\sigma$-sortable.\qedhere
  \end{itemize}
\end{proof}

The condition of Theorem~\ref{mesh_patterns_necessary} is also
sufficient for a permutation to be $\sigma$-sortable.

\begin{theorem}\label{mesh_patterns_sufficient}
  If $\pi \in \Av\left( 2314,\meshpatt \right)$, then $\pi$ is
  $\sigma$-sortable.
\end{theorem}

\begin{proof}
  Suppose, for a contradiction, that $\pi$ is not $\sigma$-sortable, that is,
  $\out{(\pi)}$ contains an occurrence of $231$. Let
  $\pi=m_1 B_1 m_2 B_2 \cdots m_k B_k$ be the ltr-minima decomposition of
  $\pi$. By Lemma~\ref{ltr_min_sortable}, we have
  $\out{(\pi)}=\widetilde{B_1} \widetilde{B_2}\cdots \widetilde{B_k} m_k
  m_{k-1} \cdots m_2 m_1$.
  Since the ltr-minima are popped from the $\sigma$-stack in increasing
  order, neither $b$ nor $c$ can be a ltr-minimum.  Suppose that $b \in B_i$
  and $c \in B_j$, for some $i \le j$. If $i<j$, then
  $m_i b m_j c \simeq 2314$, which is forbidden. Suppose instead that $i=j$
  and consider the leftmost ascent $x<y$ in $\widetilde{B_i}$ (indeed there
  is at least one ascent in $\widetilde{B_i}$, since the elements $b,c$
  constitute a noninversion in $\widetilde{B_i}$). There are two
  possibilities.

  \begin{enumerate}
  \item If $y$ comes after $x$ in $\pi$ then $x$ has to be popped before $y$
    is pushed onto the $\sigma$-stack. Therefore, when $x$ is popped, there are two
    elements $u,v$ in the $\sigma$-stack, with $v$ above $u$, such that
    $u v w \simeq 231$, where $w$ is the next element of the input. If
    $v \neq x$, then also $v$ is popped after $x$ (for the same reason), but
    this is a contradiction with the fact that $x$ and $y$ constitute an
    ascent in $\widetilde{B_i}$. Thus we have $v=x$ and $u x w \simeq 231$,
    which implies that $w \neq y$ and $u x w y \simeq 2314$ in $\pi$,
    contradicting the assumption that $\pi$ avoids $2314$.
  \item Suppose instead that $y$ precedes $x$ in $\pi$. Observe that $y$
    has to be on the $\sigma$-stack when $x$ enters, because
    $\out{(\pi)}$ contains the ascent $(x,y)$ (this fact will be
    frequently used in the sequel). In this situation, $m_i y x$ is an
    occurrence of $132$ in $\pi$. We now show that either $m_i y x$ is
    an occurrence of $\meshpatt$ or $\pi$ contains $2314$. If there is
    an element $z$ that precedes $m_i$ in $\pi$ such that $x<z<y$ (so
    that $z m_i y x \simeq 3142$), then $z$ cannot be a ltr-minimum. In
    such a case, in fact, by Lemma~\ref{ltr_min_sortable}, $z$ would be
    in the $\sigma$-stack below $y$ when $x$ is pushed, but
    $zyx \simeq 231$, which is impossible due to the restriction of the
    $\sigma$-stack. Instead, if $z \in B_\ell$ for some $\ell<i$, then
    $m_\ell z m_i y \simeq 2314$. Therefore we can assume that every
    element that precedes $m_i$ in $\pi$ is either smaller than $x$ or
    greater than $y$. Finally, suppose that there is an element $z$
    between $y$ and~$x$ in $\pi$ such that $z<x$, which gives an
    occurrence $m_i y z x$ of either $2413$ or $1423$. Then, since $y$
    is still in the $\sigma$-stack when $x$ is pushed and $z$ precedes~$x$ in
    $\pi$, $z$ enters the $\sigma$-stack above $y$, and so $\widetilde{B_I}$ contains
    either $x \ldots z \ldots y$ or $z \ldots x \ldots y$, with
    $z<x$. However, both cases give a contradiction, because $(x,y)$ is
    the first ascent in $\out{(\pi)}$.\qedhere
  \end{enumerate}
\end{proof}

\begin{cor}\label{bar_pattern_char}
  $\Sort(132)=\Av \left( 2314,\meshpatt \right)$.
\end{cor}

In accordance with Theorem~\ref{classes_and_non_classes}, the set
$\Av \left( 2314,\meshpatt \right)$ is not a permutation class; this is due
to the presence of the non-classical mesh pattern $\meshpatt$. For example,
the $\sigma$-sortable permutation $2413$ contains the pattern $132$, which is not $\sigma$-sortable.

\section{Grid decomposition of $132$-sortable
  permutations}\label{section_grid}

In this section we exploit the characterization in terms of pattern
avoidance in order to provide a geometric description of
$\Sort(\sigma)$. We start by refining the ltr-minima decomposition
$\pi=m_1 B_1 m_2 B_2 \ldots m_k B_k$ of $\pi$ as follows:
\begin{itemize}
\item for $j \ge 1$, the $j$-th \emph{vertical strip} of $\pi$ is
  $B_j$;
\item for $i \ge 1$, the $i$-th \emph{horizontal strip} of $\pi$ is
  $H_i=\left\lbrace x \in \pi :m_i < x < m_{i-1} \right\rbrace$, where
  $m_0=+ \infty$.
\item for any two indices $i,j$, the \emph{cell} of indices $i,j$ of
  $\pi$ is $C_{i,j}=H_i \cap B_j$ (note that $C_{i,j}$ is empty when
  $i>j$).
\item the \emph{core} of $\pi$ is
  $\mathcal{C}(\pi)=B_1 B_2 \ldots B_k$, obtained from $\pi$ by removing
  the ltr-minima.
\end{itemize}

In what follows, the content of each $B_j ,H_i , C_{i,j}$ will be
regarded as a permutation.  For example, let
$\pi=13 \,14 \,15 \,10 \,12 \,6 \,7 \,8 \,11 \,9 \,3 \,1 \,4 \,5 \
2$. Then (see Figure~\ref{Figure_grid_dec}):

\begin{itemize}
\item the ltr-minima of $\pi$ are $13,10,6,3,1$;
\item the vertical strips are $B_1=14 \,15 \simeq 1 \,2$,
  $B_2=12 \simeq 1$, $B_3=7 \,8 \,11 \,9 \simeq 1 \,2 \,4 \,3$,
  $B_4= \emptyset$ and $B_5=4 \,5 \,2 \simeq 2 \,3 \,1$;
\item the horizontal strips are $H_1=14 \,15 \simeq 1 \,2$,
  $H_2= 12 \,11 \simeq 2 \,1$, $H_3= 7 \,8 \,9 \simeq 1 \,2 \,3$,
  $H_4 = 4 \,5 \simeq 1 \,2$ and $H_5 = 2 \simeq 1$;
\item the nonempty cells are $C_{1,1}= 14 \,15 \simeq 1 \,2$,
  $C_{2,2}=12 \simeq 1$, $C_{2,3}=11 \simeq 1$,
  $C_{3,3}=7 \,8 \,9 \simeq 1 \,2 \,3$, $C_{4,5}=4 \,5 \simeq 1 \,2$ and
  $C_{5,5}= 2 \simeq 1$;
\item the core of $\pi$ is
  $\mathcal{C}(\pi)=14 \,15 \,12 \,7 \,8 \,11 \,9 \,4 \,5 \,2 \simeq 9 \
  10 \,8 \,4 \,5 \,7 \,6 \,2 \,3 \,1$.
\end{itemize}

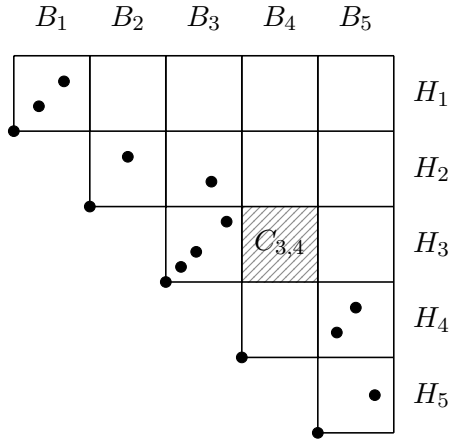
\begin{figure}
  \begin{center}
    \begin{tikzpicture}[scale=1, baseline=20.5pt]
      \fill[NE-lines] (3,2) rectangle (4,3);
      \node[scale=1] at (3.5,2.5) {$C_{3,4}$};
      \draw [semithick] (4,0) grid (5,5);
      \draw [semithick] (3,1) grid (4,5);
      \draw [semithick] (2,2) grid (3,5);
      \draw [semithick] (1,3) grid (2,5);
      \draw [semithick] (0,4) grid (1,5);
      \filldraw (0,4) circle (2pt);
      \filldraw (0.33,4.33) circle (2pt);
      \filldraw (0.66,4.66) circle (2pt);
      \filldraw (1,3) circle (2pt);
      \filldraw (1.5,3.66) circle (2pt);
      \filldraw (2,2) circle (2pt);
      \filldraw (2.2,2.2) circle (2pt);
      \filldraw (2.4,2.4) circle (2pt);
      \filldraw (2.6,3.33) circle (2pt);
      \filldraw (2.8,2.8) circle (2pt);
      \filldraw (3,1) circle (2pt);
      \filldraw (4,0) circle (2pt);
      \filldraw (4.25,1.33) circle (2pt);
      \filldraw (4.5,1.66) circle (2pt);
      \filldraw (4.75,0.5) circle (2pt);
      \node[scale=1] at (0.5,5.5) {$B_1$};
      \node[scale=1] at (1.5,5.5) {$B_2$};
      \node[scale=1] at (2.5,5.5) {$B_3$};
      \node[scale=1] at (3.5,5.5) {$B_4$};
      \node[scale=1] at (4.5,5.5) {$B_5$};
      \node[scale=1] at (5.5,4.5) {$H_1$};
      \node[scale=1] at (5.5,3.5) {$H_2$};
      \node[scale=1] at (5.5,2.5) {$H_3$};
      \node[scale=1] at (5.5,1.5) {$H_4$};
      \node[scale=1] at (5.5,0.5) {$H_5$};
\end{tikzpicture}
\caption{The grid decomposition of the permutation
  $\pi=13 \,14 \,15 \,10 \,12 \,6 \,7 \,8 \,11 \,9 \,3 \,1 \,4 \,5 \,2$. The
  image of $\pi$ under the bijection of Theorem~\ref{bij_12231} is the
  restricted growth function
  $\phi(\pi)=111223332345445$.}\label{Figure_grid_dec}
\end{center}
\end{figure}

The above terminology refers to the graphical representation of $\pi$,
see Figure~\ref{Figure_grid_dec}.  We now collect several properties of
$\sigma$-sortable permutations, in order to find a geometric description of them,
as well as their enumeration.

The next lemma provides a useful property of $\sigma$-sortable permutations.  In
spite of its simplicity, it gives a rather strong constraint on the
shape of a $\sigma$-sortable permutation.

\begin{lemma}\label{No_switch_component}
  Let $\pi$ be a $\sigma$-sortable permutation and suppose that the cell
  $C_{i,j}$ is nonempty, for some $i,j$. Then the cell $C_{u,v}$ is
  empty for each pair of indices $(u,v)$ such that $u<i$ and $v>j$.
\end{lemma}

\begin{proof}
  Suppose there are two elements $x \in C_{i,j}$ and $y \in C_{u,v}$
  such that $u<i$ and $v>j$. Then $m_i x m_v y \simeq 2314$, which is
  impossible by Theorem~\ref{mesh_patterns_necessary}.
\end{proof}

Our next results are some pattern avoidance characterizations for
$C_{i,j}, H_i$ and $\mathcal{C}(\pi )$.

\begin{lemma}\label{inversion_in_a_cell}
  Let $\pi$ be a $\sigma$-sortable permutation and suppose that the cell
  $C_{i,j}$ contains an inversion $x>y$, where $x$ precedes $y$ in
  $C_{i,j}$. Then there is an element~$z$ between $x$ and $y$ in $\pi$
  such that $z<m_i$.
\end{lemma}

\begin{proof} We refer to Figure~\ref{figure_cells} for a description of
  the statement of the lemma. For~$x$ and $y$ as above, we have
  $m_i xy \simeq 132$. In particular, $x$ and $y$ are in the same cell
  $C_{i,j}$ and $m_i$ is the corresponding ltr-minimum, hence every
  element $w$ preceding $m_i$ in $\pi$ is greater than $x$ (because
  $w > m_{i-1}$ and $x<m_{i-1}$). Therefore, as a consequence of
  Theorem~\ref{mesh_patterns_necessary}, there exists an element $z$
  between $x$ and $y$ in $\pi$ such that $z<y$. If $z<m_i$, then we are
  done. Otherwise, if $z>m_i$, we can repeat the same argument using the
  occurrence $m_i xz$ of $132$, in which we have replaced $y$ with the
  element $z$ that comes strictly before $y$ in $\pi$; continuing in
  this way we eventually find an element of $\pi$ with the desired
  property.
\end{proof}

\begin{figure}
  \begin{center}
    \begin{tikzpicture}
      \draw [semithick] (1,1) rectangle (2,2);
      \draw [dashed] (1,-1) -- (1,1);
      \draw [dashed] (0,1) -- (1,1);
      \draw [dashed] (-1,2) -- (1,2);
      \draw [dashed] (1.5,0) -- (1.5,1.5);
      \draw [thin] (1.33,1.66) -- (1.66,1.33);
      \node [label=below: $m_j$] at (1,-1) {$\bullet$};
      \node [label=left: $m_i$] at (0,1) {$\bullet$};
      \node [label=left: $m_{i-1}$] at (-1,2) {$\bullet$};
      \node [] at (2.5,1.5) {$C_{i,j}$};
      \node [] at (1.33,1.66) {$\bullet$};
      \node [] at (1.66,1.33) {$\bullet$};
      \node [] at (1.15,1.8) {$x$};
      \node [] at (1.85,1.55) {$y$};
      \node [label=right: $z$] at (1.5,0) {$\bullet$};
    \end{tikzpicture}
    \hspace{1cm}
    \begin{tikzpicture}
      \draw [semithick] (0,0) rectangle (1,1);
      \draw [semithick] (2,2) rectangle (3,3);
      \draw [dashed] (-1,0) -- (0,0);
      \draw [dashed] (2,-1) -- (2,2);
      \draw [dashed] (1,1) -- (2,2);
      \node [label=left: $m_i$] at (-1,0) {$\bullet$};
      \node [label=below: $m_v$] at (2,-1) {$\bullet$};
      \node [] at (0.5,1.5) {$C_{i,j}$};
      \node [] at (2.5,3.5) {$C_{u,v}$};
      \node [] at (0.33,0.66) {$x$};
      \node [] at (2.33,2.66) {$y$};
      \node [] at (0.5,0.5) {$\bullet$};
      \node [] at (2.5,2.5) {$\bullet$};
    \end{tikzpicture}
  \end{center}
  \caption{The constructions of Lemma~\ref{inversion_in_a_cell}, left, and of
    Lemma~\ref{No_switch_component}, right.}\label{figure_cells}
\end{figure}
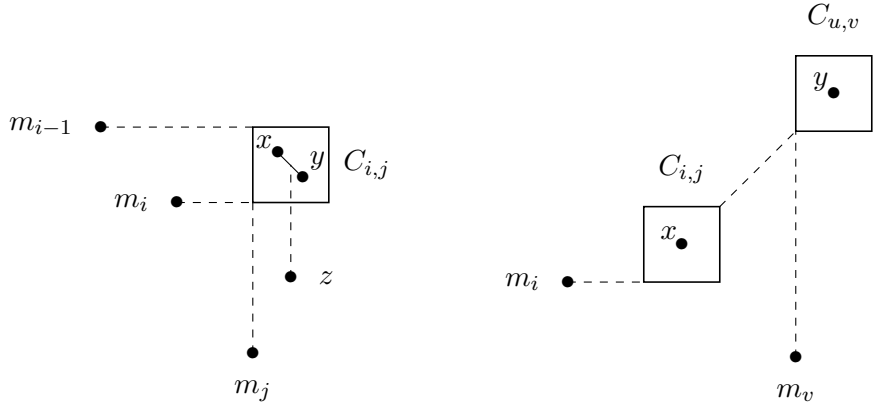

\begin{prop}\label{Layered_cells}
  If $\pi$ is a $\sigma$-sortable permutation, then $C_{i,j} \in \Av(132,213)$,
  for every $i,j$.
\end{prop}
\begin{proof}
  Suppose that $C_{i,j}$ contains an occurrence $acb$
  of $132$. By Lemma~\ref{inversion_in_a_cell}, there exists an element
  $z$ between $c$ and $b$ in $\pi$ such that $z<m_i$. In particular,
  $m_i a z b \simeq 2314$, which is a contradiction since $\pi$ is
  $\sigma$-sortable (by Theorem~\ref{mesh_patterns_necessary}). On
  the other hand, if $C_{i,j}$ contains an occurrence $bac$ of $213$,
  then $(b,a)$ is an inversion in the cell $C_{i,j}$ and therefore, again
  by Lemma~\ref{inversion_in_a_cell}, there is an element~$z$ between
  $b$ and $a$ in $\pi$ with $z<m_i$ and $m_i bzc \simeq 2314$, a
  contradiction.
\end{proof}

\begin{prop}\label{Layered_H_strips}
  If $\pi$ is a $\sigma$-sortable permutation, then $H_i \in \Av(132,213)$, for
  every $i$.
\end{prop}

\begin{proof}
  This is a consequence of Lemma~\ref{ltr_min_sortable} and
  Proposition~\ref{Layered_cells}.
\end{proof}

\begin{prop}\label{Complement_no_213}
  If $\pi$ is a $\sigma$-sortable permutation, then $\mathcal{C}(\pi) \in \Av(213)$.
\end{prop}
\begin{proof}
  Suppose $\pi$ contains an occurrence $bac$ of $213$ that does not involve
  any ltr-minimum and suppose that $b \in C_{i,j}$ for some $i,j$. Note that
  $b<c$, so, by Lemma~\ref{ltr_min_sortable}, $b$ and $c$ must belong to the
  same vertical strip $B_j$. Now, if $a \in C_{\ell,j}$, with $\ell>i$, then
  $m_i b a c \simeq 2314$, which is a contradiction, since $\pi$ is
  $\sigma$-sortable. Therefore we must have $a \in C_{i,j}$. This results in
  an occurrence $m_i b a$ of $132$, with $b$ and $a$ both in the cell
  $C_{i,j}$; thus, by Lemma~\ref{inversion_in_a_cell}, there is an element $z$
  between $b$ and $a$ in $\pi$ such that $z<m_i$ and $m_i b z c \simeq 2314$,
  which is again a contradiction.
\end{proof}

What we have established so far in this section are necessary conditions
satisfied by $\sigma$-sortable permutations. Since each prefix of a
$\sigma$-sortable permutation is still $\sigma$-sortable, removing the last
element from a $\sigma$-sortable permutation $\pi' \in \Sort_{n+1}(\sigma)$
returns a permutation $\pi \in \Sort_n(\sigma)$. In other words, every
permutation in $\Sort_{n+1}(\sigma)$ is obtained from a permutation
$\pi \in \Sort_n(\sigma)$ by inserting a new rightmost element and suitably
rescaling the remaining ones. However, not just any integers are allowed for
such an insertion. Inserting a new minimum, which corresponds to creating a
new vertical strip, is always allowed, because it cannot create any new
occurrence of $2314$ or $\meshpatt$. On the other hand, if $\pi$ has $k$
ltr-minima and we try to insert a new element in one of the cells $C_{i,k}$
of the last vertical strip, we have to obey the conditions stated in
Lemma~\ref{No_switch_component} and Propositions~\ref{Layered_H_strips}
and~\ref{Complement_no_213}. In particular,
Proposition~\ref{Layered_H_strips} implies that any permutation in~$H_i$ is
\emph{co-layered}, that is, it is the skew sum of increasing
permutations. Thus, in order to get a new co-layered permutation from a given
one, and also in order to avoid the forbidden pattern 2314, we find that
there are at most two possible ways to insert a new rightmost element in
$C_{i,k}$:

\begin{enumerate}
\item $\Do$: insert a new minimum in $C_{i,k}$ (which is also a new
  minimum of the horizontal strip $H_i$);
\item $\Co$: create a consecutive ascent in the two final positions of
  $C_{i,k}$,
\end{enumerate}
recalling that an ascent $(a,b)$ is consecutive if $b=a+1$.

This approach is formalized as follows. Let $\pi$ be a $\sigma$-sortable
permutation with $k$ ltr-minima. For $i \ge 1$, the cell $C_{i,k}$
(belonging to the last vertical strip) is said to be \emph{active} if
both of the following conditions hold:
\begin{itemize}
\item[(i)] $C_{u,v}$ is empty for each $u,v$ such that $u>i$ and $v<k$;
\item[(ii)] inserting a new rightmost element according to $\Do$ does
  not create an occurrence of $213$ in $\mathcal{C}(\pi)$. 
\end{itemize}

Note that, thanks to condition (i), condition (ii) can be equivalently
stated by saying that the permutation $\bigcup_{j\geq i+1} C_{j,k}$ is
increasing. Moreover, if a cell $C_{i,k}$ is not active, then every insertion
of a new rightmost element in $C_{i,k}$ results in a non $\sigma$-sortable permutation
due to Lemma~\ref{No_switch_component} and
Proposition~\ref{Complement_no_213}. We shall prove that if instead $C_{i,k}$
is active, then exactly one of the operations $\Do$ and $\Co$ can be
performed in order to obtain a $\sigma$-sortable permutation. To this end we
distinguish two cases, depending on whether $C_{i,k}$ is empty or not.

\begin{prop}\label{cell_legal_insertion}
  Let $\pi=\pi_1 \ldots \pi_n$ be a $\sigma$-sortable permutation with $k$ ltr-minima
  and let $C_{i,k}=\gamma_1 \ldots \gamma_t$ be a nonempty active
  cell of $\pi$. Let $x=\pi_n$ and suppose $x \in C_{\ell,k}$. Then:
\begin{enumerate}
\item performing $\Do$ on $C_{i,k}$ returns a $\sigma$-sortable permutation 
  $\pi'$ if and only if $\ell>i$;
\item performing $\Co$ on $C_{i,k}$ returns a $\sigma$-sortable permutation
  $\pi'$ if and only if $\ell \le i$.
\end{enumerate}
\end{prop}

\begin{proof}
  \begin{enumerate}
  \item Suppose $\ell<i$ and we want to insert a new rightmost element
    $\gamma_{t+1}$ into $C_{i,k}$ according to $\Do$. Assume, for a
    contradiction, that the resulting permutation $\pi'$ is $\sigma$-sortable. The
    elements $\gamma_t$ and $\gamma_{t+1}$ form an inversion in $C_{i,k}$, so
    by Lemma~\ref{inversion_in_a_cell} there exists an element $z$ between
    $\gamma_t$ and $\gamma_{t+1}$ in $\pi$ such that $z<m_i$. Hence
    $m_i \gamma_t z x \simeq 2314$, which contradicts the assumption that
    $\pi$ is $\sigma$-sortable. Instead, if $\ell=i$, that is, $\gamma_t=x=\pi_n$,
    then $\gamma_t \gamma_{t+1}$ is an inversion inside $C_{i,k}$ such that
    $\gamma_t$ and $\gamma_{t+1}$ are adjacent in $\pi$. This implies that
    $\pi$ is not $\sigma$-sortable (again as a consequence of
    Lemma~\ref{inversion_in_a_cell}).

    Conversely, suppose that $\ell>i$ and $\gamma_{t+1}$ is inserted into
    $C_{i,k}$ according to $\Do$. By
    Theorem~\ref{mesh_patterns_necessary},
    $\pi \in \Av(2314,\meshpatt)$, so we just have to show that the
    permutation $\pi'$ obtained after the insertion still avoids the two
    forbidden patterns. If $\gamma_{t+1}$ plays the role of the $2$
    in an occurrence of $132$, say $ac \gamma_{t+1}$, then we have
    either $acx \gamma_{t+1}\simeq 1423$ or
    $acx \gamma_{t+1}\simeq 2413$, which means that the selected
    occurrence of 132 is not an occurrence of the mesh pattern
    $\meshpatt$. Otherwise, suppose there is an occurrence
    $bca \gamma_{t+1}$ of $2314$ in $\pi'$. If $m_k =1$ precedes $c$ in
    $\pi$, then $c a \gamma_t \simeq 213$ in $\mathcal{C}(\pi)$,
    contradicting Proposition~\ref{Complement_no_213}. On the other hand, if
    $m_k$ follows $c$ in $\pi$, then $c \in B_j$, for some $j<k$, and
    $\gamma_t \in B_k$, with $c<\gamma_t$, contradicting
    Lemma~\ref{ltr_min_sortable}.

  \item Suppose we insert $\gamma_{t+1}$ into $C_{i,k}$ according to
    $\Co$ and $\ell>i$. Then $\gamma_t x \gamma_{t+1}$ is an occurrence of
    $213$ in $\mathcal{C}(\pi')$, hence $\pi'$ is not $\sigma$-sortable, due to
    Proposition~\ref{Complement_no_213}, as desired.

    Conversely, suppose that $\ell < i$ and we insert $\gamma_{t+1}$ into
    $C_{i,k}$ according to $\Co$; this means that
    $\gamma_{t+1}=\gamma_t+1$. The resulting permutation $\pi'$ does not
    contain an occurrence $bcad$ of $2314$ with $\gamma_{t+1}=d$, for
    otherwise $bcax$ would be an occurrence of $2314$ in $\pi$, contradicting
    the hypothesis that $\pi$ is $\sigma$-sortable. On the other hand, suppose there
    are two elements $a,c$ in $\pi$ such that $a c \gamma_{t+1}$ is an
    occurrence of $132$. We now prove that $a c \gamma_{t+1}$ is not an
    occurrence of the mesh pattern $\meshpatt$ by distinguishing two
    cases. 

If $c>m_{i-1}$ (note that $i>\ell$, so $m_{i-1}$ exists), then
    $a<\gamma_{t+1} <m_{i-1}$, so $m_{i-1}$ precedes $a$ in $\pi$ (because
    $a<m_{i-1}$ and $m_{i-1}$ is a ltr-minimum) and $m_{i-1}ac \gamma_{t+1}$
    would be an occurrence of $3142$. Instead, if $c<m_{i-1}$, then $c$ is
    not a ltr-minimum, because $a<c$ precedes $c$; moreover, $c$ is in
    $C_{i,k}$, since $c<m_{i-1}$ and $c>\gamma_{t+1}$, hence $c \gamma_t x$
    is an occurrence of $213$ in $\mathcal{C}(\pi)$, which is impossible due 
    to Proposition~\ref{Complement_no_213}. Finally, if $\ell=i$, then
    $x=\gamma_t$, $\gamma_{t+1}=\gamma_t+1$ and they are adjacent in $\pi'$,
    so $\gamma_{t+1}$ is neither part of an occurrence of $2314$ nor of
    $\meshpatt$, since otherwise $\gamma_t$ would be as well, contradicting
    the hypothesis that $\pi$ is $\sigma$-sortable.\qedhere
  \end{enumerate}
\end{proof}

If $C_{i,k}$ is empty, then the operation $\Co$ does not make sense, so the
only possibility is to try to perform $\Do$. The next proposition asserts
that this can always be done.

\begin{prop}\label{empty_cell}
  Let $\pi=\pi_1 \ldots \pi_n$ be a $\sigma$-sortable permutation with $k$ ltr-minima
  and let $C_{i,k}$ be an empty active cell of $\pi$. Let $\pi'$
  be the permutation obtained from $\pi$ by inserting a new rightmost
  element $y$ in $C_{i,k}$ according to $\Do$. Then $\pi'$ is $\sigma$-sortable.
\end{prop}

\begin{proof}
  By Theorem~\ref{mesh_patterns_necessary} we have that
  $\pi \in \Av(2314,\meshpatt)$ and we want to prove that
  $\pi' \in \Av(2314,\meshpatt)$. Suppose there are three elements
  $b,c,a$ in $\pi$ such that $bcay \simeq 2314$. Since $c>b$, the
  element $c$ is not a ltr-minimum of $\pi$. Suppose that
  $c \in C_{u,v}$, for some $u,v$. If $a$ is a ltr-minimum, then of
  course $v<k$, and we have also $u>i$, because $y$ is the minimum of
  its horizontal strip and $y>c$. This would imply that $C_{u,v}$ is a
  nonempty cell, with $u>i$ and $v<k$, which is impossible since
  $C_{i,k}$ is active. Otherwise, if $a$ is not a ltr-minimum, then
  $cay \simeq 213$ in $\mathcal{C}(\pi')$, which again
  contradicts the assumption that $C_{i,k}$ is active.

  Next, in order to prove that $\pi'$ does not contain the mesh pattern
  $\meshpatt$, suppose there are two elements $a,c$ in $\pi$ such
  that $acy \simeq 132$ and suppose $c \in B_j$, for some $j \le k$. If
  $j<k$, then $ac m_k y$ is an occurrence of $2413$, as
  desired. Otherwise, if $j=k$, we have that $c \in C_{\ell,k}$, for some
  $\ell<k$, because $C_{i,k}$ is empty before we insert $y$; moreover,
  $m_\ell$ precedes $a$ in $\pi$, because $m_\ell>y$ and $a<y$. Thus
  $m_\ell a c y \simeq 3142$, as desired.
\end{proof}

\begin{cor}\label{recursive_construction}
  Let $\pi$ be a $\sigma$-sortable permutation. Then, for every active cell of
  $\pi$, exactly one of $\Do$ and $\Co$ generates a $\sigma$-sortable
  permutation.
\end{cor}

Propositions~\ref{cell_legal_insertion} and~\ref{empty_cell} can be
interpreted as a constructive procedure to generate inductively every
$\sigma$-sortable permutation. Starting from $\pi \in \Sort_n(\sigma)$,
one can either insert a new rightmost minimum or choose an active cell
of $\pi$ and insert a new rightmost element by performing either $\Do$
or $\Co$, according to the rules of
Propositions~\ref{cell_legal_insertion} and~\ref{empty_cell}. Moreover,
if the number of active cells of $\pi$ is $t$, then $\pi$ produces $t+1$
$\sigma$-sortable permutations of length $n+1$: one for each active cell and one
when a new minimum is inserted. In principle, this gives rise to a
generating tree for $\sigma$-sortable permutations, which is often a useful tool
for enumeration. Unfortunately, we have not been able to fully
understand the succession rule of such a tree (namely, we do not know
how to compute the number of active sites of the permutations generated
by a permutation with a given number of active sites). However, by
exploiting the grid structure of $\sigma$-sortable permutations, our
generating procedure leads to a bijection with a class of pattern
avoiding {\rgfs}.

Let $\pi=\pi_1 \ldots \pi_n$ be a permutation with $k$ ltr-minima
$m_1,\dots,m_k$ and set $m_0=+\infty$. Define the map $\phi$ by setting
$\phi(\pi)=r_1 \ldots r_n$, where $r_i=j$ if
$m_{j} \le \pi_i < m_{j-1}$. In other words, the map $\phi$ scans the
permutation $\pi$ from left to right and records the index of the
horizontal strip that contains the current element of $\pi$, including
the ltr-minima in the corresponding strips. For example, if
$\pi=13 \,14 \,15 \,10 \,12 \,6 \,7 \,8 \,11 \,9 \,3 \,1 \,4 \,5 \,2$,
then $\phi(\pi)=111223332345445$ (see Figure~\ref{Figure_grid_dec}).
Note that $\phi$ is defined for any permutations.  We will now show
that, when restricted to $\sigma$-sortable permutations, the map $\phi$
is a bijection between $\Sort_n(\sigma)$ and $\mathcal{R}_n(12231)$.

\begin{theorem}\label{bij_12231}
  Let $\phi: \Sort_n(\sigma) \rightarrow \mathcal{R}_n(12231)$ be
  defined as above. Then $\phi$ is a bijection.
\end{theorem}

\begin{proof}
  By Lemma~\ref{RGF_prop}, avoiding $12231$ is equivalent to avoiding
  $2231$. We start by proving that, for each $\sigma$-sortable permutation
  $\pi$, $\phi(\pi)$ avoids $2231$, that is, $\phi$ is well-defined. Suppose,
  on the contrary, that $\phi(\pi)$ contains an occurrence
  $r_{i_1} r_{i_2} r_{i_3} r_{i_4}$ of $2231$. Consider the leftmost
  occurrence $r_j$ of the integer $r_{i_1}$ in~$\pi$ (note that $j \le
  i_1$).
  Then $r_j$ corresponds through $\phi$ to the ltr-minimum of the horizontal
  strip of index $r_{i_1}$ in $\pi$. Hence the elements
  $\pi_j \pi_{i_2} \pi_{i_3} \pi_{i_4}$ form an occurrence of $2314$ in
  $\pi$, which contradicts Theorem~\ref{mesh_patterns_necessary}.

  That $\phi$ is injective is a consequence of
  Corollary~\ref{recursive_construction}. Moreover, using the construction of
  Proposition~\ref{cell_legal_insertion}, we will show that $\phi$ is
  surjective. Given a {\rgf} $R=r_1 r_2 \ldots r_n$, construct the
  permutation $\pi_R$ by scanning $R$ from left to right and, when the
  current element is $r_\ell$, insert a new rightmost element $\pi_\ell$ in
  the following way (suitably rescaling the previous elements when
  necessary):

  \begin{itemize}
  \item when $r_\ell$ is the first occurrence of an integer in $R$ then
    $\pi_\ell =1$;
  \item otherwise, $\pi_\ell$ is inserted in the horizontal strip
    $H_{r_\ell}$, according to the rules of
    Proposition~\ref{cell_legal_insertion}.
  \end{itemize}

  We now wish to prove that, if the  {\rgf} $R$
  avoids $2231$, then $\pi_R$ is a $\sigma$-sortable permutation such that
  $\phi (\pi_R )=R$. It is easy to see that $\phi (\pi_R )=R$, as a
  direct consequence of the definition of $\phi$.  Since insertions
  inside active cells are always allowed, what remains to be shown is
  that each element is in fact inserted into an active cell. We now
  argue by contradiction, and suppose that $y$ is the first element that
  is inserted into a nonactive cell $C_{i,j}$. According to the
  definition of an active cell, there are two cases to consider.

  \begin{enumerate}
  \item If there exists a nonempty cell $C_{u,v}$, with $u>i$ and
    $v<j$, then, given any $x\in C_{u,v}$, the elements of $R$
    corresponding to $m_u x m_j y$ form an occurrence of $2231$, which
    is forbidden.
  \item Suppose that inserting a new rightmost element according to $\Do$
    creates an occurrence $bay$ of $213$ that does not involve any
    ltr-minima. Let $H_u$ be the horizontal strip that contains $b$ and let
    $H_v$ be the horizontal strip that contains $a$. Note that $v\ge u>
    i$.
    If $v>u$, then the elements corresponding to $m_u bay$ in $R$ form an
    occurrence of $2231$, which is again a contradiction. On the other hand,
    if $v=u$, then $a$ belongs to the same horizontal strip of~$b$, so,
    since $a<b$, $a$ was inserted according to $\Do$. Therefore, by
    Proposition~\ref{cell_legal_insertion} and our choice of $y$, the element
    $a'$ that precedes $a$ in $\mathcal{C}(\pi)$ belongs to $H_w$, for some
    $w>u$. As a consequence, the elements $m_u b a' c$ correspond to an
    occurrence of $2231$ in $R$, which is impossible. \qedhere
  \end{enumerate}
\end{proof}

\begin{cor}\label{enumeration_bij}
  For every natural number $n$, $|\Sort_n(\sigma)| = | \mathcal{R}_n(12231)|$.
\end{cor}

The enumeration of these {\rgfs} follows from the results in~\cite{JM}, where
it is shown that $12231$ is Wilf-equivalent to $12332$ (see
Table~\ref{table_wilf_class} here). Moreover, they also show that $1221$-avoiding
{\rgfs} are enumerated by the Catalan numbers. Hence, as a consequence of
Theorem 31 in~\cite{JM}, we immediately obtain the following formula for
$\sigma$-sortable permutations:
$$|\Sort_n(\sigma )| = \sum_{k=0}^{n-1} \binom{n-1}{k} \mathfrak{c}_k.
$$
The above sequence is A007317 in~\cite{Sl}.

\begin{table}
  \begin{center}
    \renewcommand{\arraystretch}{0.9}
    \begin{tabular}{c|c|c}
      Pattern $p$ & Formula for $|\mathcal{R}_n(p)|$ & OEIS \\
      \hline
      & & \\
      12123, 12132, 12134, 12213, & & \\
      12231, 12234, 12312, 12321,
      & $\displaystyle{\sum_{k=0}^{n-1} \binom{n-1}{k} \mathfrak{c}_k}$ & $A007317$ \\
      12323, 12331, 12332 & & \\
    \end{tabular}
  \end{center}
  \caption{The eleven patterns of the Wilf-class enumerated by
    $A007317$, see~\cite[Table~3]{JM}.}\label{table_wilf_class}
\end{table}

\section{Combinatorial proofs for pattern-avoiding restricted growth
  functions}\label{section_enum}

In the previous section we have completely solved the problem of
counting $\sigma$-sortable permutations, by explicitly finding a
bijection with the class of $12231$-avoiding
{\rgfs}, whose enumeration is known~\cite{JM}. However, this does not
provide a clear understanding of why the resulting counting sequence is
the binomial transform of Catalan numbers. What we would like to have is
a transparent bijective link between $\sigma$-sortable permutations and 
some combinatorial objects whose structure immediately reveals the
connection with this counting sequence.

The current section is devoted to illustrating some bijections involving sets
of {\rgfs} avoiding a certain pattern. Although the enumerations of these
sets are known, essentially as corollaries of the general mechanism presented
by Jel\'{\i}nek and Mansour~\cite{JM}, we provide new direct combinatorial
proofs, exhibiting links with other well studied combinatorial
structures. More precisely, we start by describing a presumably new bijection
between $\mathcal{R}_n(1221)$ and the set of Dyck paths of semilength
$n$. Moreover, for some of the patterns $p$ listed in
Table~\ref{table_wilf_class}, we describe bijections between $\mathcal{R}(p)$
and other combinatorial objects, such as labeled Motzkin paths and
pattern-avoiding permutations. Finally, we define a bijection between
$\mathcal{R}(12321)$ and $\mathcal{R}(12231)$ that, together with some of the
previous results, gives a transparent bijective argument that fully explains
the enumeration of $\sigma$-sortable permutations.

\subsection{The pattern $1221$}

The following lemma is contained in \cite{CDDGGPS} and provides a nice
characterization of $1221$-avoiding  {\rgfs}.

\begin{lemma}[\cite{CDDGGPS}, Lemma 6.2]\label{char}
  Let $R$ be a {\rgf}. Then $R \in \mathcal{R}(1221)$ if and only if the
  subword $w(R)$ obtained by removing the first occurrence of each letter in
  $R$ is weakly increasing.
\end{lemma}

As an immediate consequence, we have the following corollary.

\begin{cor}\label{active_sites_1221}
  Let $R=r_1 \ldots r_n \in \mathcal{R}(1221)$ and $M=\max(R)$. If $R$ has no
  repeated elements let $t=1$, otherwise let $t$ be the maximum among
  repeated elements of $R$.  Then $r_1 \ldots r_n j \in \mathcal{R}(1221)$ if
  and only if $t\le j\le M+1$.
\end{cor}

The previous corollary can be rephrased using the language of
\emph{generating trees} (see for instance~\cite{BDLPP}).  In particular,
we say that an integer $x$ is an \emph{active site} of the  {\rgf} $R\in \mathcal{R}(1221)$ whenever adding $x$ at the end
of $R$ returns another  {\rgf} belonging to
$\mathcal{R}(1221)$ (whose length is of course increased by one). Due to
Corollary~\ref{active_sites_1221}, the set of active sites of $R$ is the
interval $\lbrace t,t+1,\dots,M,M+1 \rbrace$ and thus there are
$M+1-t+1$ active sites, where $M$ and $t$ are as in the corollary. In
the language of generating trees, any  {\rgf}
obtained from $R$ this way is called a \emph{child} of $R$.

For the next theorem, we recall the definition of a \emph{double rise}
in a Dyck path, which is an occurrence of the consecutive pattern
$\U \U$.

\begin{theorem}\label{enum_1221}
  There is a bijection $\psi$ from $\mathcal{R}_n(1221)$ to the set of Dyck
  paths of semilength $n$, such that the maximum of $R\in\mathcal{R}_n(1221)$
  equals one plus the number of double rises in the path $\psi(R)$.  As a
  consequence, denoting by $f_{n,k}$ the number of elements in
  $\mathcal{R}_n(1221)$ whose maximum is $k$, we get that
  $f_{n,k}=\mathfrak{n}_{n,k}$, where $\mathfrak{n}_{n,k}$ is the $(n,k)$-th
  Narayana number.
\end{theorem}

\begin{proof}
  Recall from~\cite{BDLPP} that every Dyck path $\widetilde{P}$ of semilength
  $n+1$ is obtained (in a unique way) from a Dyck path $P$ of semilength $n$
  by inserting a peak $\U \D$ either before a $\D$-step in the last
  descending run of $P$ or after the last $\D$-step. This construction gives
  rise to a well known generating tree for Dyck paths, such that the number
  of active sites of a path $P$ is $k+1$, where $k$ is the length of the
  maximal suffix of $P$ entirely made of $\D$-steps. The path $\widetilde{P}$
  is therefore a child of $P$ in the associated generating tree. Our goal is
  to define (in a recursive fashion) a bijection $\alpha$ between the
  generating tree of $\mathcal{R}(1221)$ and the generating tree of Dyck
  paths. In other words, we wish to show that $\alpha$ is a bijection
  preserving both the size (that is, a {\rgf} $R\in \mathcal{R}_n(1221)$ is
  mapped to a Dyck path of semilength $n$) and the number of active sites.

  We start by setting $\alpha(1)=\U \D$. Note that $1$ has two active
  sites, since the children of $1$ are $11$ and $12$. The path
  $\U \D$ has two active sites as well, since its children are
  $\U \U \D \D$ and $\U \D \U \D$. Now let $R=r_1 \ldots r_n$ and
  $\alpha(R)=p_1 \ldots p_{2n}$, for some $n \ge 1$. Suppose that the
  number of active sites of both $R$ and $\alpha(R)$ is $k$. Let
  $M=\max(R)$ and let $t$ be the maximum element of $R$ that is not a 
  ltr-maximum of $R$. By Corollary~\ref{active_sites_1221}, the active
  sites of $R$ form the interval $\lbrace t,t+1,\dots,M,M+1 \rbrace$,
  with $M+1-t+1=k$ by hypothesis. Moreover, the length of the maximal
  suffix of $\D$-steps of $\alpha(R)$ is $k-1$. We shall
  describe $\alpha$ on the children of both $R$ and $\alpha(R)$, and
  show that the number of active sites is still preserved.
  \begin{itemize}
  \item The child of $R$ corresponding to the active site $M$ is mapped to
    the path obtained from $\alpha(R)$ by inserting a new peak $\U\D$
    immediately after the last $\D$-step of $\alpha(R)$. Here the active
    sites of the resulting sequence are $M+1-M+1=2$. The same happens for the
    resulting Dyck path, since the length of the maximal suffix of $\D$-steps
    is $1$.

  \item For $i=t,\dots,M-1$, the child of $R$ corresponding to the active
    site $i$ is mapped to the path obtained from $\alpha(R)$ by inserting a
    new peak $\U \D$ immediately before the $(M+1-i)$-th $\D$ step of the
    last descending run. The number of active sites of the resulting {\rgf}
    is then $M+1-i+1=M+2-i$, which is also the length of the maximal suffix
    of $\D$-steps of the resulting path.

  \item Finally, the child of $R$ corresponding to the active site $M+1$ is
    mapped to the path obtained from $\alpha(R)$ by inserting a new peak
    $\U\D$ immediately before the first $\D$-step of the last descending run
    of $\alpha(R)$. In this case the number of active sites of the resulting
    {\rgf} is $M+2-(t+1)=k+1$. Moreover, the number of active sites of the
    resulting path is also $k+1$, since the length of its maximal suffix of
    $\D$-steps is increased by one with respect to $\alpha(R)$.
  \end{itemize}

  Therefore $\alpha$ is a bijection between the two generating trees, as
  desired. To conclude, observe that the number of double rises in
  $\alpha(R)$ is equal to $\max(R)-1$. Indeed, by definition of $\alpha$,
  each double rise in $\alpha(R)$ corresponds to the first occurrence of an
  integer in $R$, except for the first occurrence of $1$ (which does not
  create a double rise). As is well known (see for example~\cite{De}), the
  number of Dyck paths of semilength $n$ with $k-1$ double rises is given by
  $\mathfrak{n}_{n,k}$, which gives the desired equality
  $f_{n,k}=\mathfrak{n}_{n,k}$.
\end{proof}

\pagebreak 
\begin{cor}\label{enumeration}
  Let $n \ge 0$ and $g_n= |\mathcal{R}_{n}(12332)|$. Denote by
  $g(n,k)$ the number of elements in $\mathcal{R}_n(12332)$ whose
  maximum is $k$, for $1 \le k \le n$. Then
  $$g(n+1,k+1)=
    \sum_{j=k}^{n} \binom{n}{j} \mathfrak{n}_{j,k}.
  $$
\end{cor}

\begin{proof}
  As observed in~\cite{JM}, every $12332$-avoiding {\rgf} of length $n+1$ can
  be obtained by choosing $n-j$ positions for the $1$s (except for the first
  $1$, which is fixed) and then choosing a {\rgf}
  $\widehat{R} \in \mathcal{R}_j (1221)$ for the remaining $j$ spots (where
  the elements of $\widehat{R}$ incremented by 1 will be inserted). In
  particular, if the maximum of $\widehat{R}$ is $k$, then the resulting
  {\rgf} has maximum $k+1$. So, as a consequence of Theorem~\ref{enum_1221},
  we have $g(n+1,k+1)= \sum_{j=k}^{n} \binom{n}{j} \mathfrak{n}_{j,k}$.
\end{proof}

As it turns out, the formula in Corollary~\ref{enumeration} also enumerates
$\sigma$-sortable permutations according to the number of their ltr-minima. A
proof will be given in upcoming sections (Proposition~\ref{narayana_321}
and Theorem~\ref{bij_12321_12231}) by means of a bijection between $12231$- and
$12321$-avoiding {\rgfs}. However, although we have a precise geometrical
description of $\Sort(\sigma)$, we have not been able to find a direct proof
of this.

\begin{openpr}\label{open_prob_distr}
  Prove directly (that is, without using a bijection involving different
  objects) that the number of $132$-sortable permutations of length
  $n+1$ with $k+1$ left-to-right minima is given by
  $$\sum_{j=k}^{n} \binom{n}{j} \mathfrak{n}_{j,k}.
  $$
\end{openpr}

\subsection{The patterns $12323$ and $12332$}\label{section_cont_frac}

Let
$$F(x)=
\sum_{n \ge 0} \left( \sum_{k=0}^{n-1}\binom{n-1}{k}\mathfrak{c}_k \right) x^n
$$
be the ordinary
generating function of $\sigma$-sortable permutations (or, equivalently,
of $\mathcal{R}(12323)$ and of $\mathcal{R}(12332)$). Then $F(x)$ can be
expressed using the following continued fraction (see, for
example,~\cite{B,F}):
$$F(x)=
\cfrac{1}{1-2x-\cfrac{x^2}{1-3x-\cfrac{x^2}{1-3x-\cfrac{x^2}{1-3x-\dots}}}}
$$

A nice combinatorial interpretation of this continued fraction can be given
in terms of labeled Motzkin paths, via Flajolet's general
correspondence~\cite{F}. More precisely, $|\Sort_{n+1}(\sigma)|$ is the
number of Motzkin paths of length $n$ such that each horizontal step at
height zero has two types of labels $\ell_0$, $\ell_1$ and each horizontal
step at height at least one has three types of labels
$\ell_0,\ell_1,\ell_2$. Let $\M^{lab}_n$ be the set of such labeled Motzkin
paths of length $n$. We now define a map $\beta$ from $\M^{lab}_n$ to {\rgfs}
of length $n+1$ (see Figure~\ref{Motzkin_figure}). Let $P \in \M^{lab}_n$ and
let $\Delta$ be an initially empty stack. We construct a {\rgf} $R$ by
scanning from left to right the labels of $P$ (including $\U$ and $\D$ for
upstep and downstep, respectively). We start by setting $R=1$. Then we append
a new rightmost element to $R$ according to the following rules, where $L$
denotes the currently scanned label:
\begin{itemize}
\item if $L=\U$ then append a new strict maximum $M$ and push $M$ onto
  $\Delta$;
\item if $L=\D$ then append $\top(\Delta)$ and pop it from $\Delta$;
\item if $L=\ell_0$, then append a new strict maximum (without
  pushing it onto $\Delta$);
\item if $L=\ell_1$ then append $1$;
\item if $L=\ell_2$ then append $\top(\Delta)$ (without popping
  it from $\Delta$).
\end{itemize}

In other words, $\U$ corresponds to the first occurrence of a letter $x$ that
appears at least twice in $R$, $\D$ to the last occurrence of such a letter,
and ${\ell}_2$ to an occurrence of such an $x$ that is neither the first nor
the last.  Moreover, the label $\ell_0$ corresponds to an element $x\neq 1$
appearing only once and the label $\ell_1$ corresponds to the element 1.

It is worth noting the correspondence between the labels of a Motzkin
path~$P$ described above and properties of the set partition associated (in
Section~\ref{prelim}) to the {\rgf} $R=\beta(P)$.  Namely, if $B$ is a block
of cardinality at least 2 in such a partition and $B$ doesn't contain 1, then
$\U$, $\D$ and $\ell_2$ correspond, respectively, to the least, the largest
and any of the remaining elements of the block. Furthermore,~$\ell_0$
corresponds to a singleton block not containing $1$ and~$\ell_1$ corresponds
to the elements of the block containing $1$. With this correspondence the
auxiliary stack~$\Delta$ is seen to keep track, at each stage of the
construction of $R$, of the open blocks in the corresponding partition, that
is those blocks that have not yet received all their elements.  

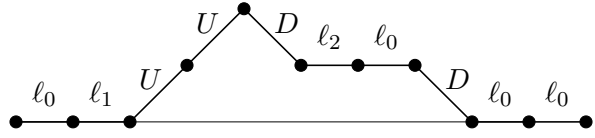
\begin{figure}
  \begin{center}
    \begin{tikzpicture}[scale=0.75, baseline=20.5pt]
      \draw [ultra thin] (0,0) -- (9,0);
      \draw [semithick] (0,0) -- (2,0);
      \draw [semithick] (2,0) -- (4,2);
      \draw [semithick] (4,2) -- (5,1);
      \draw [semithick] (5,1) -- (7,1);
      \draw [semithick] (7,1) -- (8,0);
      \draw [semithick] (8,0) -- (10,0);
      \filldraw (0,0) circle (3pt);
      \filldraw (1,0) circle (3pt);
      \filldraw (2,0) circle (3pt);
      \filldraw (3,1) circle (3pt);
      \filldraw (4,2) circle (3pt);
      \filldraw (5,1) circle (3pt);
      \filldraw (6,1) circle (3pt);
      \filldraw (7,1) circle (3pt);
      \filldraw (8,0) circle (3pt);
      \filldraw (9,0) circle (3pt);
      \filldraw (10,0) circle (3pt);
      \node[] at (0.5,0.5) {$\ell_0$};
      \node[] at (1.5,0.5) {$\ell_1$};
      \node[above,left] at (2.75,0.75) {$U$};
      \node[above,left] at (3.75,1.75) {$U$};
      \node[] at (4.75,1.75) {$D$};
      \node[] at (5.5,1.5) {$\ell_2$};
      \node[] at (6.5,1.5) {$\ell_0$};
      \node[] at (7.75,0.75) {$D$};
      \node[] at (8.5,0.5) {$\ell_0$};
      \node[] at (9.5,0.5) {$\ell_0$};
    \end{tikzpicture}
    \caption{The labeled Motzkin path corresponding, via the bijection
      $\beta$ of Theorem~\ref{bij_Motzkin}, to the restricted growth function
      $12134435367$, which in turn corresponds to the set partition
      13--2--479--56--8--10--11.}\label{Motzkin_figure}
  \end{center}
\end{figure}

\begin{theorem}\label{bij_Motzkin}
  The map $\beta$ is a bijection between $\M^{lab}_n$ and
  $\mathcal{R}_{n+1}(12323)$.
\end{theorem}
\begin{proof}
  It is straightforward to see that $\beta$ is injective and that $\beta(P)$
  is a {\rgf} for every $P \in \M^{lab}_n$. Since
  $|\M^{lab}_n|=|\mathcal{R}_n(12323)|$, we only need to show that $\beta(P)$
  avoids $12323$, for each $P \in \M^{lab}_n$. Suppose, for a contradiction,
  that $abcb'c'$ is an occurrence of $12323$ in $\beta(P)$. This implies, of
  course, that $b,c\neq 1$. Without loss of generality, we may assume 
  that $b$ and $c$ are the first occurrences of the corresponding integers in
  $\beta(P)$; then both $b$ and $c$ correspond to $\U$-steps in $P$ and are
  pushed onto $\Delta$. Moreover, since $b'=b$ and $b'$ follows $c$ in
  $\beta (P)$, when $c$ enters $\Delta$, $b$ is still in, and so $c$ lies
  above $b$ in $\Delta$. Now observe that the element $b'$ must correspond to
  either a $\D$-step or a horizontal step labeled $\ell_2$ of $P$. However,
  in both cases, when $b'$ is inserted into $\beta (P)$, $b$ has to be at the
  top of the stack, hence $c$ should have been popped. This would imply that
  there are no more occurrences of $c$ in $\beta (P)$ after $b'$, which is
  not the case, since $c'=c$.
\end{proof}

\begin{remark}\label{remark_12332_motzkin}
  If we replace the stack $\Delta$ with a queue $\Xi$, then the same map
  gives a bijection with  {\rgfs} avoiding
  $12332$. The proof is analogous to the previous one, and is omitted.
\end{remark}

\begin{remark}
  If we restrict the previous bijections to Motzkin paths with no horizontal
  steps labeled $\ell_1$, then we get bijections with {\rgfs} that avoid
  $1212$ (if we use a stack $\Delta$) or $1221$ (if we use a queue $\Xi$),
  provided that we remove the 1 at the beginning and decrease all the other
  elements by one. This follows again from the characterization of
  $\mathcal{R}(12323)$ and $\mathcal{R}(12332)$ given in~\cite{JM}. The
  corresponding continued fraction is then:
  $$
  F(x)=
  \cfrac{1}{1-x-\cfrac{x^2}{1-2x-\cfrac{x^2}{1-2x-\cfrac{x^2}{1-2x-\cdots}}}}
  $$

  This gives an alternative proof of the fact that
  {\rgfs} avoiding either $1221$ or $1212$ are enumerated by the
  Catalan numbers, whose generating function is known to be given by the
  above continued fraction.
\end{remark}

\begin{remark}\label{sumremark}
  As a consequence of the bijections in Theorem~\ref{bij_Motzkin} and
  Remark~\ref{remark_12332_motzkin}, the statistic ``sum of the numbers of
  $\U$ and $\ell_0$ steps'' in $\M^{lab}_n$ is equidistributed with the
  statistic ``(value of the) maximum minus one'' both in
  $\mathcal{R}_{n+1}(12332)$ and in $\mathcal{R}_{n+1}(12323)$. The same
  holds for the statistics ``number of labels $\ell_0$'' and ``number of
  singletons $\neq \{ 1\}$'', as well as for the statistics ``number of
  labels $\ell_1$'' and ``number of occurrences of $1$ minus one''. Some
  computations seem to suggest that the distribution of the maximum is the
  same for several other patterns of the same Wilf-class, namely
  $12123, 12132, 12213, 12231, 12312, 12321, 12331$, so we suspect that the
  same approach should lead to straightforward bijections, by suitably
  modifying the interpretation of the steps.\\
  For example, define $r_i$ to be a \emph{repeated} ltr-maximum of a {\rgf}
  $r_1r_2\ldots r_n$ if
  $r_i=\max\left\lbrace r_1,\dots,r_{i-1} \right\rbrace$.  Then steps having
  label $\ell_1$ seem to have the same distribution as the repeated
  ltr-maxima in $\mathcal{R}(12321)$ and $\mathcal{R}(12312)$, so in order to
  define a bijection with $\M^{lab}$ it could be enough to find the
  ``correct'' interpretations for steps having labels $\D$ and $\ell_2$.
\end{remark}

\subsection{The patterns $12321$ and $12312$}

In this subsection we deal with {\rgfs} avoiding the patterns $12321$ and 
$12312$, respectively, by exhibiting a connection with permutations avoiding
the patterns 321 and 312, respectively.

Let $R=r_1 \ldots r_n$ be a {\rgf}. Recall from Remark~\ref{sumremark} that
$r_i$ is said to be a repeated ltr-maximum when
$r_i=\max \left\lbrace r_1,\dots,r_{i-1} \right\rbrace$, that is, when $r_i$
is at least as great as all preceding letters, but not a ltr-maximum. Denote
by $\mathcal{R}^{n.r.}$ the set of {\rgfs} with no repeated ltr-maxima. The
notations $\mathcal{R}^{n.r.}_n$ and $\mathcal{R}^{n.r.}(Q)$, for a pattern
$Q$, are defined in the usual way. If $R=r_1\ldots r_n\in\mathcal{R}^{n.r.}$
is a {\rgf} with no repeated ltr-maxima, denote by $\widetilde{R}$ the
subsequence of $R$ obtained by deleting its ltr-maxima. Note that
$\widetilde{R}$ is not necessarily a {\rgf}. For example, if
$R=121311245246$, then $\widetilde{R}=111224$.

\begin{lemma}\label{lemma_321}
  Let $R \in \mathcal{R}^{n.r.}$. Then $R$ avoids $12321$ if and only
  $\widetilde{R}$ is weakly increasing.
\end{lemma}

\begin{proof}
  Suppose $\widetilde{R}= \ldots b a \ldots$, where $b>a$. Note that
  $b$ is not a repeated ltr-maximum of $R$, so there has to be an
  element $c$ in $R$ such that $c>b$ and $c$ comes before $b$. Then $R$
  contains an occurrence $cba$ of $321$ and therefore it also contains
  $12321$, by Lemma~\ref{RGF_prop}.

  Conversely, if $R$ contains an occurrence $abcb'a'$ of $12321$, then
  $b'$ precedes $a'$ in $\widetilde{R}$ and $b'>a'$, so $\widetilde{R}$ is not
  weakly increasing.
\end{proof}

We can now define a bijection between $\mathcal{R}^{n.r.}(12321)$ and
$\Av(321)$. In fact, the previous lemma roughly says that the combinatorial
structure of elements of $\mathcal{R}^{n.r.}(12321)$ is analogous to that of
permutations in $\Av(321)$, that is, they can both be written as a shuffle of
two weakly increasing sequences (namely, the strictly increasing sequence of
the ltr-maxima and the weakly increasing sequence of the remaining
elements). Let $R=r_1 \ldots r_n \in \mathcal{R}^{n.r.}(12321)$ and suppose
$\widetilde{R}=r_{i_1} \ldots r_{i_k}$, where $k\ge0$. Construct a
permutation of length $n$ by keeping the same positions for the ltr-maxima
and mapping $\widetilde{R}$ to a strictly increasing sequence
$S=s_1 \ldots s_k$ as follows:
\begin{itemize}
\item $s_1=r_{i_1}$;
\item $s_j=s_{j-1}+(r_{i_j}-r_{i_{j-1}})+1$, for $j \ge 2$.
\end{itemize}

Finally, in order to get a permutation that avoids $321$, insert the
remaining elements in increasing order (they will be the ltr-maxima).  For
instance, if $R=121314234$, then $\widetilde{R}=11234$, so we get $S=12468$
and the resulting permutation is
$\mathbf{3} \mathbf{5} 1 \mathbf{7} 2 \mathbf{9} 4 6 8$ (bold elements are
the ltr-maxima). Note that a {\rgf} having maximum $k$ (equivalently, with
$k$ ltr-maxima) is mapped to a permutation with $k$ ltr-maxima. It is
straightforward to prove that the resulting permutation avoids
$321$. Moreover, since $321$-avoiding permutations are uniquely determined by
positions and values of their ltr-maxima, the strictly increasing sequence
$S$ is enough to uniquely identify one such permutation. Therefore the map
defined above is injective. Finally, the construction proposed can be easily
inverted, so the map is a size-preserving bijection between
$\mathcal{R}^{n.r.}(12321)$ and $\Av(321)$. We thus have the following
result, whose proof immediately follows from the above discussion.

\begin{prop}\label{narayana_321} 
The number of  {\rgfs} in
  $\mathcal{R}_n ^{n.r.}(12321)$ is $\mathfrak{c}_n$. Moreover, the
  number of  {\rgfs} in
  $\mathcal{R}_n ^{n.r.}(12321)$ having maximum $k$ is given by
  $\mathfrak{n}_{n,k}$.
\end{prop}

Next we show that any  {\rgf} avoiding $12321$ is
obtained by choosing a sequence in $\mathcal{R}^{n.r.}(12321)$ and then
inserting some repeated ltr-maxima.

\begin{theorem}\label{thm_321}
  Let $R$ be a  {\rgf} and let $\alpha(R)$ be the
  sequence obtained from~$R$ by removing all the repeated ltr-maxima. 
  Then $\alpha(R)$ is a  {\rgf}. Moreover,
  $R$ avoids $12321$ if and only $\alpha(R)$ avoids $12321$.
\end{theorem}

\begin{proof}
  It is easy to check that $\alpha(R)$ is still a {\rgf} and clearly
  $\alpha(R)$ avoids $12321$ if $R$ does. On the other hand, suppose that $R$
  contains an occurrence $abcb'a'$ of $12321$. Note that $b'$ and $a'$ are
  not repeated ltr-maxima, so they are elements of $\alpha(R)$ and they
  follow $c$ in $R$. Let $c'$ be the first occurrence of the integer $c$ in
  $R$. Then $c' \in \alpha(R)$ and $c'$ precedes $b'$ in $\alpha(R)$, so
  $\alpha(R)$ contains an occurrence $c'b'a'$ of $321$, which is equivalent
  to containing $12321$.
\end{proof}

\begin{cor}\label{enumer_12321}
  For each $n \ge 1$, we have
  $$
    |\mathcal{R}_{n+1}(12321)|
    =\sum_{k=0}^n \binom{n}{k} \mathfrak{c}_k.
  $$
  Moreover, there are
  $\sum_{j=k}^{n} \binom{n}{j} \mathfrak{n}_{j,k}$
   {\rgfs} in $\mathcal{R}_{n+1}(12321)$ with maximum
  $k$.
\end{cor}

\begin{proof}
  This is a direct consequence of the results proved in this subsection,
  together with the fact that the first element of a {\rgf} cannot be a
  repeated ltr-maximum.
\end{proof}

\begin{remark}
  The same approach can be used to find a bijection between
  $\mathcal{R}^{n.r.}(12312)$ and $\Av(312)$. In fact, $312$-avoiding
  permutations are also uniquely determined by the positions and values of
  their ltr-maxima, and a completely analogous argument can be applied.  As a
  consequence, we also have
  $$|\mathcal{R}_{n+1}(12312)|=\sum_{k=0}^n \binom{n}{k} \mathfrak{c}_k.
  $$
\end{remark}

\subsection{A bijection between $\mathcal{R}(12321)$ and
  $\mathcal{R}(12231)$}

In Section~\ref{section_grid} we showed that $\sigma$-sortable permutations
are in bijection with {\rgfs} avoiding $12231$. Although the labeled Motzkin
path approach described in Section~\ref{section_cont_frac} could be fruitful,
a direct combinatorial enumeration for the pattern $12231$ seems to be rather
more complicated than for the patterns treated in the previous section. Here
we illustrate a bijection between $\mathcal{R}(12231)$ and
$\mathcal{R}(12321)$, thus obtaining an independent proof of the enumeration
of $\Sort(\sigma)$.

From now on we say that $r_{i_1} r_{i_2} r_{i_3}$ is an occurrence of the
pattern $\widetilde{2}31$ in $R$ if $r_{i_1} r_{i_2} r_{i_3}$ is an
occurrence of $231$ and $r_{i_1}$ is not a ltr-maximum of $R$ (that
is,~$r_{i_1}$ is not the first occurrence of the corresponding integer). Note
that $\mathcal{R}(12231)=\mathcal{R}(\widetilde{2}31)$ and also
$\mathcal{R}(12321)=\mathcal{R}(321)$, so we can focus on the patterns
$\widetilde{2}31$ and $321$ instead of $12231$ and $12321$,
respectively. Given a {\rgf} $R=r_1 \ldots r_n$, define
$\rmost(R,321)=i_1 i_2 i_3$, where $r_{i_i} r_{i_2} r_{i_3}$ is the
lexicographically rightmost occurrence of $321$ in $R$. In other words, for
any other occurrence $r_{j_i} r_{j_2} r_{j_3}$ of $321$, we must have either
$j_1<i_1$, or $j_1=i_1$ and $j_2<i_2$, or $j_1=i_1$, $j_2=i_2$ and
$j_3<i_3$. If $R$ avoids $321$, set $\rmost(321)= 000$ by
convention. Similarly, denote by $\lmost(R,\widetilde{2}31)=i_1 i_2 i_3$ the
lexicographically leftmost occurrence of $\widetilde{2}31$ in $R$. If $R$
avoids $\widetilde{2}31$, set $\lmost(R,\widetilde{2}31)=(n+1)(n+1)(n+1)$.

Now, let $R=r_1 \ldots r_n \in \mathcal{R}(\widetilde{2}31)$, a hypothesis we
will assume throughout the rest of this section. Define recursively a {\rgf}
$\gamma(R)$ as follows.

\begin{enumerate}
\item $R^{(0)}=R$.
\item For $t \ge 0$, if $R^{(t)}$ contains $321$, then $R^{(t+1)}$ is 
  obtained from $R^{(t)}$ by exchanging the elements $r_{i_1}$ and
  $r_{i_2}$, where $i_1 i_2 i_3=\rmost(R^{(t)},321)$.
\item Finally, define $\gamma(R)=R^{(k)}$, where $k$ is the minimum
  index such that $R^{(k)}$ avoids $321$.
\end{enumerate}

It is easy to verify that, at each step, $R^{(t)}$ is a  {\rgf}; moreover $R^{(k)}$ avoids $321$ by construction. Thus,
in order to prove that the map $\gamma$ is well defined, we have to show
that the integer $k$ indeed exists. This follows from the next lemma.

\begin{lemma}\label{lex_321}
  For every $t \ge 0$, we have $\rmost(R^{(t+1)},321) <_\ell \rmost(R^{(t)},321)$,
  where~$<_\ell$ denotes the lexicographical order.
\end{lemma}

\begin{proof}
  Let $R^{(t)}=r^{(t)}_1 \ldots r^{(t)}_n$ and, similarly,
  $R^{(t+1)}=r^{(t+1)}_1 \ldots r^{(t+1)}_n$. Moreover, let
  $\rmost(R^{(t)},321)=i_1 i_2 i_3$ and $\rmost(R^{(t+1)},321)=j_1 j_2 j_3$.
  Note that, as illustrated in Figure~\ref{figure_321_231}, our hypothesis
  imposes some constraints on the elements of $R^{(t)}$. More precisely,
  $r^{(t)}_j \le r^{(t)}_{i_2}$, for each $j=i_1+1,\dots,i_2-1$. Also, for
  each $j=i_2+1,\dots,i_3-1$, either $r^{(t)}_j \le r^{(t)}_{i_3 }$ or
  $r^{(t)}_j \ge r^{(t)}_{i_1}$. Finally, $r^{(t)}_j \ge r^{(t)}_{i_2}$ for
  each $j>i_3$. We will repeatedly use these inequalities throughout this
  proof. Our goal is now to show that $j_1 j_2 j_3 <_\ell i_1 i_2 i_3$.
  Suppose, by contradiction, that $j_1 j_2 j_3 \ge_\ell i_1 i_2
  i_3$. Consider the following case analysis.
  \begin{itemize}
  \item Suppose $j_1 > i_1$. If $j_1 < i_2$, then necessarily
    $r^{(t+1)}_{j_1}=r^{(t)}_{j_1} \le r^{(t)}_{i_2}$, due to the above
    constraints. Hence we must have $j_2 ,j_3 \neq i_2$, since otherwise the
    indices $j_1 ,j_2 ,j_3$ would not correspond to an occurrence of 321 in
    $R^{(t+1)}$. This implies that
    $r^{(t+1)}_{j_1} r^{(t+1)}_{j_2} r^{(t+1)}_{j_3}=r^{(t)}_{j_1}
    r^{(t)}_{j_2} r^{(t)}_{j_3}$
    is an occurrence of $321$ in $R^{(t)}$ as well, with
    $j_1 j_2 j_3 >_\ell i_1 i_2 i_3$: this is a contradiction, since we are
    assuming that $\rmost(R^{(t)},321)=i_1 i_2 i_3$. Next suppose that
    $j_1=i_2$ (and so $j_2 > i_2$). Note that
    $r^{(t)}_{i_1}=r^{(t+1)}_{i_2}=r^{(t+1)}_{j_1}$, hence
    $r^{(t)}_{i_1} r^{(t)}_{j_2} r^{(t)}_{j_3}$ is an occurrence of $321$ in
    $R^{(t)}$ with $i_1 j_2 j_3 >_\ell i_1 i_2 i_3$, which is
    impossible. Finally, suppose that $j_1>i_2$. Then obviously
    $r^{(t)}_{j_1} r^{(t)}_{j_2} r^{(t)}_{j_3}=r^{(t+1)}_{j_1}
    r^{(t+1)}_{j_2} r^{(t+1)}_{j_3}$
    is an occurrence of $321$ in $R^{(t)}$, with
    $j_1 j_2 j_3 >_\ell i_1 i_2 i_3$, again a contradiction.
  \item Suppose instead that $j_1=i_1$ and $j_2>i_2$. Then
    $r^{(t+1)}_{i_1}=r^{(t)}_{i_2}$ and $j_2 > i_2$, so
    $r^{(t)}_{i_2} r^{(t)}_{j_2} r^{(t)}_{j_3}$ is an occurrence of $321$ in
    $R^{(t)}$, with $i_2 j_2 j_3 >_\ell i_1 i_2 i_3$, which is impossible.
  \item Finally, the case $j_1=i_1$ and $j_2=i_2$ is clearly impossible,
    since we have
    $r^{(t+1)}_{i_1}=r^{(t)}_{i_2} < r^{(t)}_{i_1}=r^{(t+1)}_{i_2}$. \qedhere
  \end{itemize}
\end{proof}

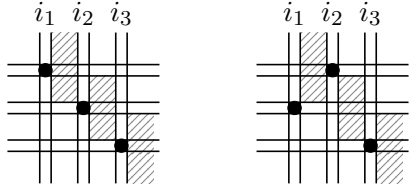
\begin{figure}
  \begin{center}
    \begin{tikzpicture}[scale=0.5, baseline=19pt]
      \fill[NE-lines] (1.15,2.15) rectangle (1.85,4);
      \fill[NE-lines] (2.15,1.15) rectangle (2.85,2.85);
      \fill[NE-lines] (3.15,0) rectangle (3.85,1.85);
      \draw [semithick] (0,0.85) -- (4,0.85);
      \draw [semithick] (0,1.15) -- (4,1.15);
      \draw [semithick] (0,1.85) -- (4,1.85);
      \draw [semithick] (0,2.15) -- (4,2.15);
      \draw [semithick] (0,2.85) -- (4,2.85);
      \draw [semithick] (0,3.15) -- (4,3.15);
      \draw [semithick] (0.85,0) -- (0.85,4);
      \draw [semithick] (1.15,0) -- (1.15,4);
      \draw [semithick] (1.85,0) -- (1.85,4);
      \draw [semithick] (2.15,0) -- (2.15,4);
      \draw [semithick] (2.85,0) -- (2.85,4);
      \draw [semithick] (3.15,0) -- (3.15,4);
      \filldraw (1,3) circle(5pt);
      \filldraw (2,2) circle (5pt);
      \filldraw (3,1) circle (5pt);
      \node[] at (1,4.5){$i_1$};
      \node[] at (2,4.5){$i_2$};
      \node[] at (3,4.5){$i_3$};
    \end{tikzpicture}
    \hspace{1cm}
    \begin{tikzpicture}[scale=0.5, baseline=19pt]
      \fill[NE-lines] (1.15,2.15) rectangle (1.85,4);
      \fill[NE-lines] (2.15,1.15) rectangle (2.85,2.85);
      \fill[NE-lines] (3.15,0) rectangle (3.85,1.85);
      \draw [semithick] (0,0.85) -- (4,0.85);
      \draw [semithick] (0,1.15) -- (4,1.15);
      \draw [semithick] (0,1.85) -- (4,1.85);
      \draw [semithick] (0,2.15) -- (4,2.15);
      \draw [semithick] (0,2.85) -- (4,2.85);
      \draw [semithick] (0,3.15) -- (4,3.15);
      \draw [semithick] (0.85,0) -- (0.85,4);
      \draw [semithick] (1.15,0) -- (1.15,4);
      \draw [semithick] (1.85,0) -- (1.85,4);
      \draw [semithick] (2.15,0) -- (2.15,4);
      \draw [semithick] (2.85,0) -- (2.85,4);
      \draw [semithick] (3.15,0) -- (3.15,4);
      \filldraw (1,2) circle (5pt);
      \filldraw (2,3) circle (5pt);
      \filldraw (3,1) circle (5pt);
      \node[] at (1,4.5){$i_1$};
      \node[] at (2,4.5){$i_2$};
      \node[] at (3,4.5){$i_3$};
    \end{tikzpicture}
  \end{center}
  \caption{On the left, the rightmost occurrence of the pattern $321$ in
    $R^{(t)}$, with indices $i_1 i_2 i_3$. Shaded boxes correspond to
    forbidden regions. On the right, the resulting pattern in $R^{(t+1)}$,
    obtained by exchanging the elements in positions $i_1$ and
    $i_2$.}\label{figure_321_231}
\end{figure}

Next we show that $\gamma$ is a bijection by proving that the recursive
construction defined above can be reversed. More precisely, $R^{(t)}$
can obtained from $R^{(t+1)}$ by transforming the leftmost occurrence of
$\widetilde{2}31$ into an occurrence of $321$ (see
Figure~\ref{diagram_gamma}).

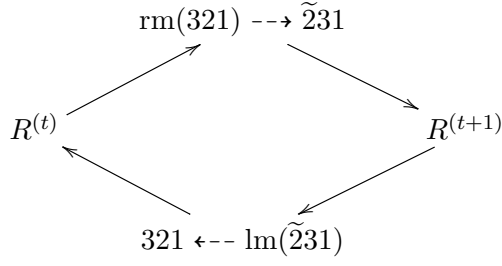
\begin{figure}[h]
  $$
  \xymatrix{
    & \rmost(321) \dashrightarrow \widetilde{2}31 \ar@{<-}[dl] \ar@{->}[dr] & \\
    R^{(t)} \ar@{<-}[dr] & & R^{(t+1)} \ar@{->}[dl] \\
    & 321 \dashleftarrow \lmost(\widetilde{2}31) &
  }
  $$
  \caption{The diagram of
    Lemma~\ref{inverse_gamma}.}\label{diagram_gamma}
\end{figure}

\begin{lemma}\label{inverse_gamma}
  Let $t\ge0$. Let $\rmost(R^{(t)},321)=i_1 i_2 i_3$ and
  $\lmost(R^{(t+1)},\widetilde{2}31)=j_1 j_2 j_3$. Then $i_1=j_1$ and
  $i_2=j_2$.
\end{lemma}

\begin{proof}
  We again refer to Figure~\ref{figure_321_231} for an illustration of the
  constraints imposed on the elements of $R^{(t)}$ by the position of the
  rightmost occurrence of $321$ inside $R^{(t)}$. We proceed by induction on
  $t$.

  Suppose first that $t=0$, that is, $R^{(0)}=r^{(0)}_1 \ldots r^{(0)}_n$
  avoids $\widetilde{2}31$, but contains $321$. Set
  $R^{(1)}=r^{(1)}_1 \ldots r^{(1)}_n$, $\rmost(R^{(0)},321)=i_1 i_2 i_3$ and
  $\lmost(R^{(1)},\widetilde{2}31)=j_1 j_2 j_3$. Note that
  $r^{(1)}_{i_1} r^{(1)}_{i_2} r^{(1)}_{i_3}$ is an occurrence of
  $\widetilde{2}31$ in $R^{(1)}$. Indeed, by Lemma~\ref{RGF_prop}, the first
  occurrence of the integer $r^{(0)}_{i_2}$ in $R^{(0)}$ precedes
  $r^{(0)}_{i_1}$, since $r^{(0)}_{i_1}>r^{(0)}_{i_2}$. Therefore
  $j_1 j_2 j_3 \le_{\ell} i_1 i_2 i_3$. We have to show that $i_1=j_1$ and
  $i_2=j_2$. Suppose, to the contrary, that $j_1 <i_1$. If either $j_2=i_1$
  or $j_2=i_2$, then $r^{(0)}_{j_1} r^{(0)}_{i_1} r^{(0)}_{j_3}$ would be an
  occurrence of $\widetilde{2}31$ in $R^{(0)}$, which is impossible since
  $R^{(0)}\in \mathcal{R}(\widetilde{2}31)$. Thus we must have $j_2 \neq i_1$
  and $j_2 \neq i_2$. In particular, since $j_2 \neq i_2$, we must have
  either $j_3=i_1$ or $j_3=i_2$, otherwise
  $r^{(0)}_{j_1} r^{(0)}_{j_2} r^{(0)}_{j_3}=r^{(1)}_{j_1} r^{(1)}_{j_2}
  r^{(1)}_{j_3}$
  would be an occurrence of $\widetilde{2}31$ in $R^{(0)}$ as well. However,
  if either $j_3=i_1$ or $j_3=i_2$, then
  $r^{(0)}_{j_1} r^{(0)}_{j_2} r^{(0)}_{i_2}$ would be an occurrence of
  $\widetilde{2}31$ in $R^{(0)}$, which is again a contradiction. Therefore
  it has to be $i_1=j_1$. Finally, the case $j_1=i_1$ and $j_2<i_2$ is
  forbidden, due to the restrictions depicted in
  Figure~\ref{figure_321_231}. Summing up, we must have $i_1=j_1$ and
  $i_2=j_2$, as desired.

  Now suppose that $t \ge 1$. Let $R^{(t)}=r^{(t)}_1 \ldots
  r^{(t)}_n$. For the rest of this proof, we fix the following
  notation:
  \begin{itemize}
  \item[-] $\rmost(R^{(t-1)},321)=t_1 t_2 t_3$;
  \item[-] $\lmost(R^{(t)},\widetilde{2}31)=s_1 s_2 s_3$;
  \item[-] $\rmost(R^{(t)},321)=i_1 i_2 i_3$;
  \item[-] $\lmost(R^{(t+1)},\widetilde{2}31)=j_1 j_2 j_3$.
  \end{itemize}
  By the inductive hypothesis we have $s_1=t_1$ and $s_2=t_2$. Moreover,
  Lemma~\ref{lex_321} implies that $t_1 t_2 t_3 >_{\ell} i_1 i_2 i_3$, hence
  $t_1 t_2 \ge_{\ell} i_1 i_2$ and $s_1 s_2 \ge_{\ell} i_1 i_2$. Note that
  $r^{(t+1)}_{i_1} r^{(t+1)}_{i_2} r^{(t+1)}_{i_3}$ is an occurrence of
  $\widetilde{2}31$ in $R^{(t+1)}$, so we must have
  $j_1 j_2 j_3 \le_{\ell} i_1 i_2 i_3$. Our goal is to show that $i_1=j_1$
  and $i_2=j_2$. We shall proceed by contradiction, so we assume that
  $j_1 <i_1$ or $j_2 <i_2$. Our strategy consists in finding an occurrence of
  $\widetilde{2}31$ in $R^{(t)}$ such that the indices of its first two
  elements strictly precede $i_1 i_2$ (in the lexicographical order). Indeed,
  this would imply that $s_1 s_2 <_{\ell} i_1 i_2$, since
  $s_1 s_2 s_3 = \lmost(R^{(t)},\widetilde{2}31)$, which is impossible since
  we know that $s_1 s_2 \ge_{\ell} i_1 i_2$.

  Suppose first that $j_1 < i_1$. If
  $\{ j_2 ,j_3 \} \cap \{ i_1 ,i_2 \} =\emptyset$, then
  $r^{(t)}_{j_1} r^{(t)}_{j_2} r^{(t)}_{j_3}$ is the desired occurrence
  of $\widetilde{2}31$ in $R^{(t)}$, since in this case $j_1,j_2,j_3$ are
  not involved in the transition from $R^{(t)}$ to $R^{(t+1)}$ and we
  are assuming that $j_1 < i_1$. Therefore at least one of $j_2$
  and $j_3$ must coincide with either $i_1$ or $i_2$. We will now show
  that, in each case, we are able to find an occurrence of $\widetilde{2}31$
  in $R^{(t)}$ with the desired property.
  \begin{itemize}
  \item If $j_2=i_1$, then
    $r^{(t+1)}_{j_2}=r^{(t+1)}_{i_1}<r^{(t)}_{i_1}$, hence
    $r^{(t)}_{j_1} r^{(t)}_{j_2} r^{(t)}_{j_3}$ is an occurrence of
    $\widetilde{2}31$ in $R^{(t)}$, and $j_1 j_2 <_{\ell} i_1 i_2$.
  \item If $j_2=i_2$, then $r^{(t)}_{j_1} r^{(t)}_{i_1} r^{(t)}_{j_3}$
    is an occurrence $\widetilde{2}31$ in $R^{(t)}$, and
    $j_1 i_1 <_{\ell} i_1 i_2$.
  \item If $j_3=i_1$, then $r^{(t)}_{j_1} r^{(t)}_{j_2} r^{(t)}_{i_2}$
    is an occurrence of $\widetilde{2}31$ in $R^{(t)}$, and
    $j_1 j_2 <_{\ell} i_1 i_2$.
  \item If $j_3=i_2$, then
    $r^{(t)}_{i_2} < r^{(t)}_{i_1} = r^{(t+1)}_{i_2}$, hence
    $r^{(t)}_{j_1} r^{(t)}_{j_2} r^{(t)}_{i_2}$ is an occurrence of
    $\widetilde{2}31$ in $R^{(t)}$, and $j_1 j_2 < i_1 i_2$.
  \end{itemize}
  The above case-by-case analysis shows that $i_1=j_1$. Moreover, we cannot
  have $j_2 < i_2$; this is again due to the restrictions illustrated in
  Figure~\ref{figure_321_231}.
\end{proof}

\begin{theorem}\label{bij_12321_12231}
  The map $\gamma$ is a size-preserving bijection between
  $\mathcal{R}(12321)$ and $\mathcal{R}(12231)$. Moreover, $\gamma$
  preserves the maximum value of a  {\rgf}.
\end{theorem}

By Theorem~\ref{bij_12321_12231} and Corollary~\ref{enumer_12321}, the
distribution of the maximum letter in {\rgfs} over $\mathcal{R}_n(12231)$ is
given by $\sum_{i=k}^{n} \binom{n}{i} \mathfrak{n}_{i,k}$. This provides a
combinatorial (even if not direct) proof of the formula stated in Open
Problem~\ref{open_prob_distr} for the distribution of ltr-minima of
$\Sort(132)$.

\section{Final remarks and future work}\label{section_final_remarks}

In Sections~\ref{section_mesh} and~\ref{section_grid} we have characterized
the elements of the set $\Sort(132)$, thus solving one of the open problems
for pattern-avoiding machines introduced in~\cite{CCF}. For three
remaining patterns $\sigma$ of length $3$, namely $213$, $231$ and $312$, a
characterization of the $\sigma$-sortable permutations remains to be found,
as well as their enumeration. The pattern $231$ seems to be significantly
more challenging than the others. This is arguably due to what seems to be
the case, according to computational evidence, namely that the 231-machine
can sort more permutations of length $n$, for each $n\ge3$, than the machines
associated to any other pattern of length $3$ (in particular, it is the only
one that can sort every permutation of length $3$).

The enumeration of $132$-sortable permutations has been obtained by means of
a bijection with {\rgfs} avoiding $12231$, whose enumeration can be found
in~\cite{JM} as an application of a much more general mechanism. In
Section~\ref{section_enum} we have found new direct proofs for related
classes of {\rgfs}, exhibiting connections with other well known
combinatorial objects, such as lattice paths and pattern-avoiding
permutations. In particular, the bijection with labeled Motzkin paths in
Theorem~\ref{bij_Motzkin} seems amenable to being extended and generalized,
in order to cover the enumeration of many patterns in the same
Wilf-equivalence class.

\end{document}